\makeatletter
\def\input@path{{els-cas-templates/}}
\makeatother
\documentclass[a4paper,fleqn]{cas-sc}

\usepackage[numbers,sort&compress]{natbib}
\usepackage{amsthm}
\usepackage{float}
\floatstyle{plain}
\restylefloat{figure}
\restylefloat{table}

\newtheorem{remark}{Remark}
\newtheorem{assumption}{Assumption}
\newtheorem{theorem}{Theorem}
\newtheorem{lemma}{Lemma}
\newtheorem{corollary}{Corollary}
\newtheorem{proposition}{Proposition}

\begin{document}
\let\WriteBookmarks\relax
\def\floatpagepagefraction{1}
\def\textpagefraction{.001}

\shorttitle{Direct and Indirect PINNs for Dirichlet Boundary Control}
\shortauthors{N.T. Quang et al.}

\title[mode=title]{Direct and Indirect Physics-Informed Neural Networks for
Dirichlet Boundary Control of Semilinear Parabolic Equations: A Conditional Error Analysis}

\author[1]{Nguyen Thanh Quang}
\fnmark[1]
\credit{Conceptualization, Methodology, Software, Formal analysis,
Writing -- original draft}

\author[1]{Ta Thi Thanh Mai}
\cormark[1]
\ead{mai.tathithanh@hust.edu.vn}
\credit{Conceptualization, Supervision, Validation,
Writing -- review and editing}

\author[1]{Bui Xuan Dieu}
\credit{Conceptualization, Supervision, Validation,
Writing -- review and editing}

\affiliation[1]{organization={Faculty of Mathematics and Informatics, Hanoi University of Science and Technology},
                addressline={1 Dai Co Viet},
                city={Hanoi},
                country={Vietnam}}

\fntext[1]{First author}
\cortext[1]{Corresponding author}

\begin{abstract}
We study physics-informed neural networks (PINNs) for the Dirichlet boundary control of a semilinear parabolic equation with Tikhonov regularization. 
Two approaches are considered. A direct PINN parameterizes the state and control by separate networks and minimizes a penalized form of the tracking objective. 
An indirect PINN instead represents the state, adjoint, and control by unconstrained networks trained jointly to satisfy the first-order optimality system, with the state-control coupling and the homogeneous adjoint boundary and terminal conditions imposed as soft penalty terms rather than enforced architecturally. For the indirect formulation we develop an error estimation framework that decomposes the total error into approximation, optimization, quadrature, and soft boundary/terminal-constraint contributions. Under standing assumptions on optimal-solution regularity and compatibility, network approximability, uniform H\"older control of the soft-constraint residuals, and a local neighborhood of the reference optimality-system solution, we derive a quantitative linearized stability estimate and a conditional local nonlinear residual-to-error estimate, and construct a computable residual indicator with a conditional reliability bound.
Numerical experiments on two manufactured test problems - a cubic reaction-diffusion equation and a linear equation with a nontrivial boundary control and adjoint - illustrate the behavior of the direct and indirect formulations.
\end{abstract}



\maketitle

\section{Introduction}

Optimal control of partial differential equations (PDEs) arises throughout
science and engineering, including the regulation of temperature,
concentration, and other diffusive processes. In Dirichlet boundary control,
the control prescribes the state trace itself and therefore models actuation
applied directly at the boundary. For semilinear parabolic equations, this
leads to analytical difficulties that are less pronounced for distributed
control: low-regularity boundary data may yield only a transposition solution,
and the first-order stationarity condition involves the normal derivative of
the adjoint. These features affect well-posedness, regularity, and numerical
approximation \cite{lasiecka2000control,troltzsch2010optimal,
aradaraymond2002dirichlet,gong2016dirichlet}.

PDE-constrained optimal control problems are commonly treated by
discretize-then-optimize or optimize-then-discretize methods based on finite
elements or finite differences and adjoint-based gradients. These methods are
mature and well understood, but require a mesh and may become costly when
geometric complexity or localized solution features demand substantial
refinement. Physics-informed neural networks (PINNs) \cite{raissi2019physics,pham2025physics}
offer a meshfree alternative in which differential equations, boundary and
initial conditions, and objective terms are sampled at collocation points.
Avoiding an explicit mesh does not remove the numerical difficulty: accuracy
still depends on network expressivity, nonconvex optimization, loss
balancing-for which adaptive gradient-based schemes such as GradNorm have
been explored in PINN settings~\cite{phuong2026gradnorm}-and the design of
the collocation set.

PINNs have been extended to PDE-constrained optimal control, with
Mowlavi and Nabi \cite{mowlavi2023optimal} studying formulations based on
penalizing the objective and PDE constraints and comparing them with
adjoint-based optimization. Two broad strategies are relevant here.
A \emph{direct} PINN
parameterizes the state and control and minimizes the tracking objective
augmented by penalties for the state equation and its initial and boundary
conditions. An \emph{indirect} PINN additionally represents the adjoint and
trains the three networks jointly against the first-order Karush-Kuhn-Tucker
(KKT) system: the state equation, adjoint equation, and stationarity condition.
Barry-Straume et al.\ \cite{barrystraume2026control} incorporate necessary
optimality conditions into the network architecture and loss to learn
states, controls, and adjoints within a single framework.
Zhang et al.\ \cite{zhang2026pinns} recently compared these approaches for
distributed control of semilinear parabolic equations with homogeneous
Neumann boundary conditions; we use their comparison as a starting point.

The theoretical analysis of PINNs provides several tools for assessing such
approximations. Shin et al.\ \cite{shin2020convergence} establish convergence
results for linear second-order elliptic and parabolic PDEs under suitable
regularity and sampling assumptions. Mishra and Molinaro
\cite{mishra2023generalization} use stability estimates for the underlying
PDE to bound the solution error in terms of training errors and quadrature
accuracy. For linear parabolic Kolmogorov equations, De Ryck and Mishra
\cite{deryck2022kolmogorov} analyze the existence of networks with small
residuals, the resulting mean-square solution error, and the relation between
continuous and empirical residuals under random sampling. These studies
motivate distinguishing approximation, optimization, and sampling errors
and identifying the stability estimate needed to convert residual control
into solution accuracy.

Error analysis is also available for neural methods in PDE control.
Dai et al.\ \cite{dai2025optimality} analyze a neural solver based on the
optimality system of linear and semilinear elliptic control problems, with
and without box constraints. They derive mean-square error bounds for the state,
control, and adjoint in terms of network complexity and sample sizes.
Garc\'ia-Cervera et al.\ \cite{garcia2023control} study PINNs for PDE
controllability, including Dirichlet boundary control of the heat equation,
and obtain state and control error estimates using observability
inequalities and energy estimates. Their objective of steering a system
to a prescribed terminal state differs from the tracking problem with
Tikhonov regularization considered here.
Quang and Mai~\cite{quang2026tikhonov} similarly combine Tikhonov
regularization with PINNs for terminal-state tracking, but for
\emph{distributed} controls, using it to select a minimum-energy control
among the infinitely many that attain the same terminal state, and derive
a control-space error estimate under a source-condition assumption; their
framework does not cover boundary actuation, which we address here.

These results provide relevant precedents, while the present analysis
addresses the coupled optimality system of a semilinear parabolic tracking
problem with Dirichlet boundary actuation. The control acts through a
boundary trace, the stationarity residual involves the normal derivative
of the adjoint, and a fully soft formulation introduces separate residuals
for the state-control coupling and the adjoint boundary and terminal
conditions. Applying the residual-to-error strategy in this setting
requires estimates that propagate these residuals through the state,
adjoint, and stationarity equations, with their respective trace norms
kept explicit. This is the focus of the conditional stability analysis
developed below.

The present work makes two contributions. First, we extend the
direct-indirect PINN framework for PDE-constrained optimal control to
semilinear parabolic problems with Dirichlet boundary actuation. We
represent the control by a single network restricted to the entire lateral
boundary and impose the state-control coupling, the homogeneous adjoint
boundary condition, and the adjoint terminal condition through soft residual
penalties. Subject to the standing regularity assumptions, the formulation and
conditional error analysis are stated for arbitrary fixed spatial dimension
$d\geq1$, while the numerical demonstrations are restricted to one space
dimension.

Second, we derive a quantitative stability estimate for the linearized
soft-constrained KKT system and extend it, in a local neighborhood of a regular
reference KKT point, to a nonlinear residual-to-error estimate. This analysis
identifies the separate contributions of network approximation, optimization,
quadrature, and soft boundary/terminal constraints. It also leads to a computable
residual indicator with a conditional reliability estimate. The estimate
retains the additive contribution of a strong adjoint boundary/terminal
residual bound and is not claimed to be either globally valid or
efficient. Two manufactured one-dimensional examples examine a cubic problem
with zero optimal control and a linear problem with nonzero control and
adjoint.

The remainder of the paper is organized as follows. Section~\ref{sec:problem} formulates the optimal control problem. Section~\ref{sec:method} presents the direct and indirect PINN formulations. Section~\ref{sec:error-analysis} develops the error estimation framework, including the standing assumptions, the quantitative stability result, a conditional a priori estimate, and a computable residual indicator with conditional reliability. Section~\ref{sec:experiments} reports the numerical experiments, and Section~\ref{sec:conclusions} concludes. Appendix~\ref{app:cited-results} collects the auxiliary analytical and approximation results used in Section~\ref{sec:error-analysis}.
\section{Problem formulation}
\label{sec:problem}

Let $\Omega\subset\mathbb{R}^d$ ($d\geq1$) be a bounded domain with
$C^{2+\beta}$ boundary $\partial\Omega$ for some $\beta\in(0,1)$, let $T>0$,
and set $Q:=\Omega\times(0,T)$ and $\Sigma:=\partial\Omega\times(0,T)$. We
consider the unconstrained Dirichlet boundary-control problem of minimizing
\begin{equation}
\label{eq:cost}
J(y,u)=\frac12\int_Q |y-y_d|^2\,dx\,dt
       +\frac{\alpha}{2}\int_\Sigma |u|^2\,ds\,dt,
\end{equation}
subject to
\begin{equation}
\label{eq:state}
\begin{cases}
\partial_t y-\nu\Delta y+f(y)=g, & \text{in }Q,\\
y=u, & \text{on }\Sigma,\\
y(x,0)=y_0(x), & \text{in }\Omega.
\end{cases}
\end{equation}
over $u\in L^2(\Sigma)$. Here $\nu>0$ is the diffusion coefficient,
$g\in L^2(Q)$ is a prescribed source, $y_0\in L^2(\Omega)$ is the initial
datum, $y_d\in L^2(Q)$ is the desired trajectory, and $\alpha>0$ is the
Tikhonov regularization parameter. The function $f:\mathbb R\to\mathbb R$ is
the reaction nonlinearity; its analytical assumptions are stated in (H3).
We write $|\Omega|$, $|Q|$, and $|\Sigma|$ for the corresponding volume,
space-time, and lateral surface-time measures. 
For $L^2(\Sigma)$ boundary data, the state
equation is understood in the transposition (very weak) sense - i.e.\ defined
by duality against a regular dual problem rather than by the usual
variational formulation - so its state belongs at least to $L^2(Q)$;
see Appendix~\ref{app:transposition}. 
The source $g$ is fixed problem data, not an optimization variable, and the case $g=0$
is included. When $d=1$, integration over
$\partial\Omega=\{0,1\}$ uses the zero-dimensional Hausdorff (counting)
measure. Existence of the reference optimality-system solution and the
additional regularity needed for the analysis are stated explicitly in
(H1)-(H2) of Section~\ref{sec:standing-assumptions}.

\section{Methodology}
\label{sec:method}

We compare a direct PINN, which minimizes a penalized form of the constrained
objective, with an indirect PINN, which solves the first-order optimality
system. The state and control approximations, their state-equation residual,
and the losses enforcing the state equation and its initial and boundary data
are common to both formulations.

\subsection{Common Neural Approximations and Residual Terms}

The state and control are approximated by $Y_\theta(x,t)$ and $U_\psi(x,t)$,
with trainable parameters $\theta$ and $\psi$. The control network is defined
on $\overline Q$, and the approximate boundary control is its restriction
$U_\psi|_\Sigma$. Let $N_{\rm int}$, $N_{\rm bc}$, and $N_{\rm ic}$ be the
numbers of interior, boundary, and initial collocation points, denoted by
$(x_j,t_j)\in Q$, $(x_j^b,t_j^b)\in\Sigma$, and $x_i\in\Omega$,
respectively. Define the state-equation residual once for both methods by
\begin{equation}
\mathcal R_y:=\partial_tY_\theta-\nu\Delta Y_\theta+f(Y_\theta)-g
\qquad\text{in }Q.
\label{eq:res_y}
\end{equation}
The three shared empirical constraint losses are
\begin{align}
\mathcal L_y&:=\frac1{N_{\rm int}}\sum_{j=1}^{N_{\rm int}}|\mathcal R_y(x_j,t_j)|^2,
\label{eq:loss_y_common}\\
\mathcal L_{\rm bc}^y&:=\frac1{N_{\rm bc}}\sum_{j=1}^{N_{\rm bc}}|Y_\theta(x_j^b,t_j^b)-U_\psi(x_j^b,t_j^b)|^2,
\label{eq:loss_bc_y_common}\\
\mathcal L_{\rm ic}^y&:=\frac1{N_{\rm ic}}\sum_{i=1}^{N_{\rm ic}}|Y_\theta(x_i,0)-y_0(x_i)|^2.
\label{eq:loss_ic_y_common}
\end{align}

\subsection{Direct PINN Formulation}

The direct loss is
\begin{equation}
\label{eq:direct_loss}
\mathcal L_{\rm dir}
=w_{\rm res}\mathcal L_y
+w_{\rm bc}\mathcal L_{\rm bc}^y
+w_{\rm ic}\mathcal L_{\rm ic}^y
+\frac{|Q|}{2}\mathcal L_{\rm track}
+\frac{\alpha|\Sigma|}{2}\mathcal L_{\rm reg},
\end{equation}
where $w_{\rm res},w_{\rm bc},w_{\rm ic}>0$ are penalty weights and
\begin{equation}
\mathcal L_{\rm track}:=\frac1{N_{\rm int}}\sum_{j=1}^{N_{\rm int}}|Y_\theta(x_j,t_j)-y_d(x_j,t_j)|^2.
\end{equation}
The control regularization is approximated using the same boundary collocation
points, sampled uniformly with respect to normalized surface-time measure on
$\Sigma$:
\begin{equation}
\mathcal L_{\rm reg}:=\frac1{N_{\rm bc}}\sum_{j=1}^{N_{\rm bc}}
|U_\psi(x_j^b,t_j^b)|^2.
\end{equation}
Thus $|Q|\mathcal L_{\rm track}$ and $|\Sigma|\mathcal L_{\rm reg}$ are
Monte-Carlo approximations of the two unnormalized integrals in
\eqref{eq:cost}; in particular, the factors $1/2$ are preserved and $\alpha$
retains its role as the Tikhonov regularization parameter. For non-uniform
sampling, the corresponding
importance weights replace the measure factors above. The three constraint
losses remain normalized empirical averages, with their relative scaling
controlled by $w_{\rm res},w_{\rm bc},w_{\rm ic}$.
This formulation is simple and avoids an adjoint network, but boundary satisfaction depends on the penalty weight and can be delicate~\cite{zhang2026pinns}.

\subsection{Indirect PINN Formulation with Soft Constraints}

\subsubsection{First-order optimality system}

For smooth variables, incorporate $y|_\Sigma=u$ and $y(\cdot,0)=y_0$ into
the admissible class and introduce the adjoint variable (Lagrange multiplier)
$\lambda$ and the Lagrangian
\begin{equation}
\ell(y,u,\lambda)=J(y,u)+\int_Q\lambda\left(\partial_ty-\nu\Delta y+f(y)-g\right)\,dx\,dt,
\end{equation}
Let $n$ be the outward unit normal on $\partial\Omega$ and write
$\partial_n v:=\nabla_xv\cdot n$ for the normal derivative. Thus admissible variations satisfy
$\delta y|_\Sigma=\delta u$ and $\delta y(\cdot,0)=0$. Integrating by parts
in time and twice in space gives
\begin{align*}
D_{y,u}\ell(y,u,\lambda)[\delta y,\delta u]
={}&\int_Q\bigl(y-y_d-\partial_t\lambda-\nu\Delta\lambda
+f'(y)\lambda\bigr)\delta y\,dx\,dt\\
&+\int_\Omega\lambda(x,T)\delta y(x,T)\,dx
-\nu\int_\Sigma\lambda\,\partial_n\delta y\,ds\,dt\\
&+\int_\Sigma\bigl(\alpha u+\nu\partial_n\lambda\bigr)
\delta u\,ds\,dt.
\end{align*}
Variation with respect to $\lambda$, followed by the independent terminal,
normal-derivative, and control variations in the display above, yields
\begin{align}
\partial_ty-\nu\Delta y+f(y)-g&=0 &&\text{in }Q, \label{eq:state_opt}\\
-\partial_t\lambda-\nu\Delta\lambda+f'(y)\lambda+(y-y_d)&=0 &&\text{in }Q, \label{eq:adjoint}\\
\alpha u+\nu\partial_n\lambda&=0 &&\text{on }\Sigma, \label{eq:stationarity}\\
y=u,\qquad \lambda&=0 &&\text{on }\Sigma, \label{eq:bc_y}\\
y(x,0)=y_0(x),\qquad \lambda(x,T)&=0 &&\text{in }\Omega. \label{eq:terminal}
\end{align}
For $\Omega=(0,1)$, $\partial_n\lambda(0,t)=-\lambda_x(0,t)$ and $\partial_n\lambda(1,t)=\lambda_x(1,t)$.

\subsubsection{Network parameterization}

In addition to the common state and control networks, the adjoint variable is
approximated by an unconstrained neural network
$\Lambda_\phi:\overline Q\to\mathbb R$ with trainable parameters $\phi$.
The networks $Y_\theta$, $U_\psi$, and $\Lambda_\phi$ do not enforce the
state-control boundary coupling, the adjoint boundary condition, or the
adjoint terminal condition by construction; these conditions are imposed
through the loss terms below. As in the direct PINN, the state initial
condition is also treated as a soft constraint.

\subsubsection{Indirect PINN loss}

Define the additional interior and boundary residuals
\begin{align}
\mathcal R_\lambda&:=-\partial_t\Lambda_\phi-\nu\Delta\Lambda_\phi+f'(Y_\theta)\Lambda_\phi+Y_\theta-y_d
&&\text{in }Q, \label{eq:res_lambda}\\
\mathcal R_u&:=\alpha U_\psi+\nu\partial_n\Lambda_\phi &&\text{on }\Sigma. \label{eq:res_u}
\end{align}
Let $N_{\rm T}$ denote the number of terminal collocation points
$x_i^T\in\Omega$. The indirect loss is
\begin{equation}
\label{eq:indirect_loss}
\mathcal L_{\rm ind}=w_y\mathcal L_y+w_\lambda\mathcal L_\lambda+w_u\mathcal L_u+w_{\rm ic}\mathcal L_{\rm ic}^y
+w_{\rm bc}^y\mathcal L_{\rm bc}^y+w_{\rm bc}^\lambda\mathcal L_{\rm bc}^\lambda+w_{\rm T}^\lambda\mathcal L_{\rm T}^\lambda,
\end{equation}
where all seven weights are strictly positive, and
\begin{align}
\mathcal L_\lambda&:=\frac1{N_{\rm int}}\sum_{j=1}^{N_{\rm int}}|\mathcal R_\lambda(x_j,t_j)|^2,
\label{eq:adjoint_residual_loss}\\
\mathcal L_u&:=\frac1{N_{\rm bc}}\sum_{j=1}^{N_{\rm bc}}|\mathcal R_u(x_j^b,t_j^b)|^2, \label{eq:boundary_residual_loss}\\
\mathcal L_{\rm bc}^\lambda&:=\frac1{N_{\rm bc}}\sum_{j=1}^{N_{\rm bc}}|\Lambda_\phi(x_j^b,t_j^b)|^2, \label{eq:bc_lambda_indirect}\\
\mathcal L_{\rm T}^\lambda&:=\frac1{N_{\rm T}}\sum_{i=1}^{N_{\rm T}}|\Lambda_\phi(x_i^T,T)|^2. \label{eq:terminal_lambda}
\end{align}
Here $(x_j^b,t_j^b)\in\Sigma$ are sampled from the chosen boundary-time
distribution, so $\mathcal L_u$, $\mathcal L_{\rm bc}^y$, and
$\mathcal L_{\rm bc}^\lambda$ are empirical boundary averages in any fixed
spatial dimension $d\geq1$. In implementation,
$\partial_n\Lambda_\phi(x,t)=\nabla_x\Lambda_\phi(x,t)\cdot n(x)$. The loss $\mathcal L_{\rm T}^\lambda$ is the sample mean over
$N_{\rm T}$ points on the terminal slice $t=T$.

For the one-dimensional experiments, $\Omega=(0,1)$ and the boundary sampling
uses the normalized counting measure on $\partial\Omega=\{0,1\}$. With
$N_{\rm bc}$ sampled times $t_j^b$ at each endpoint, define
\begin{equation}
\begin{aligned}
u_0(t)&:=U_\psi(0,t), &
\mathcal R_{u,0}(t)&:=\alpha U_\psi(0,t)-\nu\partial_x\Lambda_\phi(0,t),\\
u_1(t)&:=U_\psi(1,t), &
\mathcal R_{u,1}(t)&:=\alpha U_\psi(1,t)+\nu\partial_x\Lambda_\phi(1,t).
\end{aligned}
\label{eq:stationarity_traces}
\end{equation}
In this paired-endpoint convention, every empirical boundary average used above -
$\mathcal L_{\rm bc}^y$ and $\mathcal L_{\rm reg}$ in the direct formulation and
$\mathcal L_u$, $\mathcal L_{\rm bc}^y$, and $\mathcal L_{\rm bc}^\lambda$ in the
indirect formulation-is evaluated with denominator $2N_{\rm bc}$ and the sum of
its values at $x=0$ and $x=1$; this is exactly the implementation used in
Section~\ref{sec:experiments}.
None of the boundary coupling $y=u$, the adjoint boundary condition $\lambda=0$, or the adjoint terminal condition $\lambda(x,T)=0$ is satisfied by construction; $\mathcal L_{\rm bc}^y$, $\mathcal L_{\rm bc}^\lambda$, and $\mathcal L_{\rm T}^\lambda$ enforce them as soft penalties, in the same spirit as the common soft state constraints $\mathcal L_{\rm bc}^y$ and $\mathcal L_{\rm ic}^y$.

\section{Error Analysis}
\label{sec:error-analysis}

In this section, we establish an error estimation framework for the Indirect PINN method of Section~\ref{sec:method}.
We decompose the total error into approximation, optimization, quadrature, and soft boundary/terminal-constraint contributions, derive a conditional a priori error bound, and construct a computable a posteriori error indicator consistent with the meshfree, collocation-based nature of PINNs.

Throughout this section, the optimality system is the one stated in
\eqref{eq:state_opt}-\eqref{eq:terminal}. The common state and control
approximations are described in Section~\ref{sec:method}, and their residual
and empirical constraint losses are defined in
\eqref{eq:res_y}-\eqref{eq:loss_ic_y_common}. The adjoint approximation is
described in Section~\ref{sec:method}, and its additional differential
residuals and Indirect PINN loss are defined in
\eqref{eq:res_lambda}-\eqref{eq:terminal_lambda}.
Let $(y^*,u^*,\lambda^*)$ be a fixed reference local KKT point satisfying
(H1), where $y^*$ is the reference state, $u^*$ is the reference boundary
control, and $\lambda^*$ is the reference adjoint variable.

\subsection{Functional Setting and Notation}
\label{sec:notation}

\paragraph{Lebesgue and Sobolev conventions.}
For a domain $D\subset\mathbb R^m$, $s\ge0$, and $p\in[1,\infty]$, we use
$L^p(D)$ and $W^{s,p}(D)$ for the standard Lebesgue and Sobolev spaces
(with the usual fractional-space interpretation when $s$ is noninteger).
As usual,
\[
W^{0,p}(D)=L^p(D),\qquad H^s(D):=W^{s,2}(D),
\]
and $H_0^1(D)$ is the closure of $C_c^\infty(D)$ in $H^1(D)$. The dual of
$H_0^1(\Omega)$ is denoted by
\[
H^{-1}(\Omega):=(H_0^1(\Omega))',
\]
with duality pairing
$\langle\cdot,\cdot\rangle_{H^{-1}(\Omega),H_0^1(\Omega)}$. We also use the
standard Bochner notation $L^p(0,T;X)$ for Banach-valued functions and, for
$r\geq0$, $H^r(0,T;X)=W^{r,2}(0,T;X)$ when $X$ is Hilbert. We write
$C([0,T];X)$ for the space of continuous $X$-valued functions. For
$k\in\mathbb N_0$, $C^k(S)$ denotes the space of functions with continuous
derivatives up to order $k$ on the indicated domain or manifold $S$, with its
standard supremum norm; in particular, $C(S):=C^0(S)$. The standard homogeneous
parabolic energy space is
\[
\mathbb W(0,T):=\left\{v\in L^2(0,T;H_0^1(\Omega)):
\partial_t v\in L^2(0,T;H^{-1}(\Omega))\right\}.
\]
Unless explicitly called
\emph{parabolic} or \emph{anisotropic}, $W^{s,p}(Q)$ refers to the isotropic
Sobolev space in all $d+1$ space-time variables. In particular, the
$W^{2,\infty}(Q)$ and $W^{s,\infty}(D)$ norms used in (H4) and
Appendix~\ref{app:approximation} belong to this standard scale.

For nonnegative quantities $A$ and $B$, we write $A\lesssim B$ if
$A\leq CB$ for a constant $C>0$ that may depend on fixed problem data and
fixed loss weights, but is independent of the network resolution, training
tolerance, and collocation sample sizes.

\paragraph{Averaged residual norms.} Here and below, $\bullet$ denotes any of
the seven residual terms in the indirect loss. Each $\mathcal L_\bullet$ in
the trainable loss is a Monte-Carlo quadrature approximation of a mean squared
residual. Accordingly, for residuals we use the normalized norms
\begin{equation}
\|v\|_{L^2_{\rm av}(Q)}^2:=\frac1{|Q|}\int_Q|v|^2,\qquad
\|z\|_{L^2_{\rm av}(\Sigma)}^2:=\frac1{|\Sigma|}\int_\Sigma|z|^2,\qquad
\|w\|_{L^2_{\rm av}(\Omega)}^2:=\frac1{|\Omega|}\int_\Omega|w|^2,
\label{eq:averaged_residual_norms}
\end{equation}
where $|\Sigma|$ denotes the $(d)$-dimensional surface measure of the lateral boundary $\Sigma$. The loss terms in Section~\ref{sec:method} are defined as empirical averages over collocation points on $Q$, $\Sigma$, and $\Omega$ with the same normalization, so that the discrete losses directly approximate these averaged continuous norms. They are equivalent to $L^2(Q)$, $L^2(\Sigma)$, and $L^2(\Omega)$, respectively, up to fixed domain-dependent constants.

\paragraph{Reference solution, errors, and trace residuals.}
The superscript $*$ distinguishes the reference KKT variables from their
neural approximations. Define
$e_y:=y^*-Y_\theta$, $e_\lambda:=\lambda^*-\Lambda_\phi$,
$e_u:=u^*-U_\psi|_\Sigma$, and
$e_{y,0}:=y_0-Y_\theta(\cdot,0)$. The initial, boundary, and terminal
soft-constraint residuals are
\[
R^y_{\rm ic}:=Y_\theta(\cdot,0)-y_0=-e_{y,0},\qquad
R^y_{\rm bc}:=(Y_\theta-U_\psi)|_\Sigma,
\]
\[
R^\lambda_{\rm bc}:=\Lambda_\phi|_\Sigma,\qquad
R^\lambda_T:=\Lambda_\phi(\cdot,T).
\]
Here ``ic'' and ``bc'' denote initial- and boundary-condition residuals,
respectively, while the subscript $T$ denotes a terminal-condition residual.
Our notation consistently uses $\mathcal R_\bullet$ for differential residual
functions, $R_\bullet$ for trace residual functions, and $\mathcal L_\bullet$
for empirical mean-squared losses. The error norm is
\begin{equation}
\label{eq:E-norm}
\| (e_y, e_\lambda, e_u) \|_{\mathcal{E}}^2 := \| e_y \|_{L^2(Q)}^2 + \| e_\lambda \|_{L^2(0,T;H^1(\Omega))}^2 + \| e_u \|_{L^2(\Sigma)}^2.
\end{equation}

\begin{remark}[On the $e_y$-component of $\|\cdot\|_{\mathcal E}$]
\label{rem:E-norm-fix}
The $e_y$-component is measured in $L^2(Q)$, not
$L^2(0,T;H^1(\Omega))$. This is the strongest norm for $e_y$ provable from
the hypotheses below. The boundary datum $e_u-R^y_{\rm bc}$ is controlled
only in $L^2(\Sigma)$ and therefore determines a state only in the
transposition sense, namely in $L^2(Q)$; see
Corollary~\ref{cor:transposition} (to be stated in
Appendix~\ref{app:transposition} below). No analogous loss occurs for
$e_\lambda$: its boundary and terminal data are controlled in the strong
Hölder norms of (H5), rather than merely in $L^2(\Sigma)$ and
$L^2(\Omega)$, respectively; see Theorem~\ref{thm:quant-stability}.
\end{remark}

\paragraph{Anisotropic parabolic Sobolev space.} $H^{2,1}(Q) := L^2(0,T;H^2(\Omega))\cap H^1(0,T;L^2(\Omega))$, norm $\|v\|_{H^{2,1}(Q)}^2 := \|v\|_{L^2(0,T;H^2(\Omega))}^2 + \|\partial_t v\|_{L^2(Q)}^2$.

\paragraph{Parabolic H\"older spaces.} For $\beta\in(0,1)$, $C^{\beta,\beta/2}(\overline Q)$ is H\"older continuity of order $\beta$ in $x$, $\beta/2$ in $t$; $C^{2+\beta,1+\beta/2}(\overline Q)$ additionally requires second spatial derivatives and the first time derivative in $C^{\beta,\beta/2}(\overline Q)$.

\paragraph{Trace H\"older norms for boundary residuals.} For a function $z$ on $\Sigma$, $\|z\|_{C^{2+\beta,1+\beta/2}(\Sigma)}$ denotes the H\"older norm of $z$ as a function on the lateral boundary manifold (i.e.\ of $z(\cdot,t)$ jointly in the tangential space variable and $t$); for a function $w$ on $\Omega\times\{T\}$, $\|w\|_{C^{2+\beta}(\overline\Omega)}$ is the usual elliptic H\"older norm. These are the norms in which Schauder theory controls inhomogeneous boundary/terminal data (Lemma~\ref{lem:schauder}, to be stated in Appendix~\ref{app:schauder} below).

\paragraph{Poincar\'e constant.} $C_\Omega>0$: $\|v\|_{L^2(\Omega)}^2\leq C_\Omega\|\nabla v\|_{L^2(\Omega)}^2$ for $v\in H_0^1(\Omega)$.

\paragraph{Roles of the spaces.}
The function spaces serve distinct purposes. The space $L^2(\Sigma)$ is the
baseline control space and is also used to measure boundary residuals. The
space $L^2(Q)$ accommodates the transposition state and the interior PDE
residuals, while $L^2(\Omega)$ is used for initial and terminal residuals.
These spaces also provide the corresponding components of the error norm
$\mathcal E$.

The spaces $H_0^1(\Omega)$, $H^{-1}(\Omega)$, and $\mathbb W(0,T)$ form the
standard variational energy setting, whereas $H^{2,1}(Q)$ is used for
parabolic maximal regularity. The parabolic H\"older spaces encode the
classical Schauder regularity assumed for the reference solution and the
strong trace control in (H5). Finally, the isotropic $W^{k,\infty}$ spaces are
used for residual-compatible neural-network approximation and for the
algebraic approximation result of Appendix~\ref{app:approximation}.

\subsection{Standing Assumptions}
\label{sec:standing-assumptions}

\begin{itemize}
\item[(H1)] \textbf{Reference KKT point.} The unconstrained problem of
Section~\ref{sec:problem}, with $u\in L^2(\Sigma)$, admits a local optimal pair
$(y^*,u^*)$ and an adjoint $\lambda^*$. The state equation is understood in
the transposition sense described in Appendix~\ref{app:transposition}, the
adjoint equation is satisfied in the standard variational sense, and the
stationarity relation holds in $L^2(\Sigma)$. This is a standing hypothesis,
not a new existence theorem claimed in this paper. Under (H2), all equations
and traces in \eqref{eq:state_opt}-\eqref{eq:terminal} are classical.
\item[(H2)] \textbf{Reference-solution regularity and compatibility.} $\Omega\subset\mathbb R^d$ has $C^{2+\beta}$ boundary for some $\beta\in(0,1)$; $y_0\in C^{2+\beta}(\overline\Omega)$ and $g,y_d\in C^{\beta,\beta/2}(\overline Q)$. In addition, the (a priori unknown) reference optimal control satisfies
\[
u^*\in C^{2+\beta,1+\beta/2}(\Sigma),
\]
and the state data satisfy the compatibility conditions
\[
y_0=u^*(\cdot,0),\qquad
\partial_tu^*(\cdot,0)=\nu\Delta y_0-f(y_0)+g(\cdot,0)
\quad\text{on }\partial\Omega.
\]
Thus (H2) is explicitly an assumption on the reference optimal solution as well as on the primitive data; we do not derive this regularity of $u^*$ from the baseline space $L^2(\Sigma)$.
\item[(H3)] \textbf{Nonlinearity assumptions.} $f\in C^2(\mathbb R)$, $f'(s)\geq -c_f$ for some $c_f\geq0$, and $f''$ is bounded on bounded sets. These assumptions include both $f(y)=y^3$, with $c_f=0$, and $f(y)=0$, with $c_f=0$, used in Section~\ref{sec:experiments}.
\item[(H4)] \textbf{Residual-compatible network approximability.} Let $\widetilde u^*$ be a fixed ambient extension of $u^*$ from $\Sigma$ to $\overline Q$. The network classes, indexed by a resolution parameter $N\in\mathbb N$, contain feedforward $\tanh$ networks for which there is a sequence $\epsilon_{\rm net}(N)\to0$ satisfying
\[
\inf_{\theta,\psi,\phi}
\bigl(
\|y^*-Y_\theta\|_{W^{2,\infty}(Q)}
+\|\lambda^*-\Lambda_\phi\|_{W^{2,\infty}(Q)}
+\|\widetilde u^*-U_\psi\|_{W^{1,\infty}(Q)}
\bigr)
\le \epsilon_{\rm net}(N).
\]
Here ``isotropic'' has the meaning fixed in Section~\ref{sec:notation}:
$W^{2,\infty}(Q)$ controls all space-time derivatives of total order at most
two and therefore, in particular, the first time derivative and the second
spatial derivatives appearing in the PDE residuals. The control requires only
$W^{1,\infty}(Q)$ because its residual contains at most first derivatives.
Thus (H4) also implies approximation in the weaker error norm $\mathcal E$.
An algebraic rate is asserted only under the additional
$W^{s,\infty}$-extension assumptions stated in
Proposition~\ref{prop:approx-rate-restated} (to be stated in
Appendix~\ref{app:approximation} below).
\item[(H5)] \textbf{Uniform H\"older control of the boundary/terminal residuals.} Along the sequence of trained networks under consideration, the three soft trace residuals defined above satisfy $\|R^y_{\rm bc}\|_{C^{2+\beta,1+\beta/2}(\Sigma)}\le M_{\rm bc}$, $\|R^\lambda_{\rm bc}\|_{C^{2+\beta,1+\beta/2}(\Sigma)}\le M_{\rm bc}$, $\|R^\lambda_T\|_{C^{2+\beta}(\overline\Omega)}\le M_{\rm bc}$, for some $M_{\rm bc}<\infty$ independent of network size. At the terminal-boundary corner $\partial\Omega\times\{T\}$, the boundary datum and the terminal datum for the adjoint equation must agree, as required by Lemma~\ref{lem:schauder}. More precisely, $R^\lambda_{\rm bc}(\cdot,T)$ denotes the continuous limit of the boundary trace $R^\lambda_{\rm bc}(\cdot,t)$ as $t\uparrow T$, and
\[
R^\lambda_{\rm bc}(x,T)
=\Lambda_\phi(x,T)
=R^\lambda_T(x)
\qquad\text{for every }x\in\partial\Omega.
\]
Thus, this corner-compatibility condition is automatic here: both residual data are restrictions of the same smooth network $\Lambda_\phi$; it is not an additional independent assumption on the trained network.
Only the terms $R^\lambda_{\rm bc},R^\lambda_T$ are used in their strong H\"older form below; $R^y_{\rm bc}$ is used only through its weaker $L^2(\Sigma)$ norm.
\item[(H6)] \textbf{Local nonlinear neighborhood.} For the trained triples to which the nonlinear estimate is applied, there is a radius $\rho>0$ such that
\[
\|Y_\theta-y^*\|_{L^\infty(Q)}+\|\Lambda_\phi-\lambda^*\|_{L^\infty(Q)}\le\rho.
\]
Moreover, there exists a compact interval $I\subset\mathbb R$ containing the
ranges of $y^*$ and all trained states under consideration such that $f''$ is
Lipschitz on $I$. This local hypothesis selects the branch of the reference
KKT point; it is used only in Theorem~\ref{thm:nonlinear-stability} and its
consequences.
\end{itemize}

\begin{remark}[On hypothesis (H5)]
\label{rem:h5}
    (H5) is the price of removing the hard-constraint architecture. It is
modeled on the uniform H\"older-seminorm assumptions in
\cite[Assumption~3.2, item~3]{shin2020convergence}. For a fixed trained
network, finiteness of the displayed trace norms is automatic because
$\tanh\in C^\infty(\mathbb R)$ and the domains are compact; moreover,
\cite[Lemma~2.1]{shin2020convergence} gives convergence of all network
derivatives in $C^k(\overline Q)$ when the parameters converge and remain
uniformly bounded. What (H5) adds is a bound $M_{\rm bc}$ uniform in network
    size and training. Such a bound is not implied by the losses
    \eqref{eq:loss_bc_y_common} and
    \eqref{eq:bc_lambda_indirect}-\eqref{eq:terminal_lambda}, which control only the
applicable $L^2_{\rm av}(Q)$, $L^2_{\rm av}(\Sigma)$, or
$L^2_{\rm av}(\Omega)$ norms. Weight decay and spectral-norm control may help
enforce this property during training, while explicit verification of the
residual H\"older norms after training may assess whether it holds; these
procedures are not investigated in the numerical experiments. In
Theorem~\ref{thm:quant-stability}, $M_{\rm bc}$ is
therefore a genuine additive constant; its contribution vanishes only if
$M_{\rm bc}$ itself is driven to zero.
\end{remark}

\begin{remark}
Hypothesis (H1) deliberately separates the paper's error analysis from the harder existence theory for semilinear parabolic Dirichlet control with $L^2(\Sigma)$ data. The latter requires a transposition state formulation and is not established by the regularity arguments below; see Appendix~\ref{app:existence} and \cite{aradaraymond2002dirichlet} for context.
\end{remark}

\begin{proposition}[Regularity of the optimality system]
\label{prop:regularity}
Under (H1)-(H3), the reference triple of (H1) satisfies
\begin{equation}
y^*,\ \lambda^* \in C^{2+\beta,1+\beta/2}(\overline Q),\qquad
u^*=y^*|_\Sigma\in C^{2+\beta,1+\beta/2}(\Sigma),
\label{eq:schauder-regularity}
\end{equation}
for some $\beta\in(0,1)$. In particular, the outward normal derivative
\[
\partial_n\lambda^*(x,t):=n(x)\cdot\nabla\lambda^*(x,t),
\qquad (x,t)\in\Sigma,
\]
is well defined pointwise and belongs to
$C^{1+\beta,(1+\beta)/2}(\Sigma)$. Moreover, $y^*$ and $\lambda^*$ are
bounded on $\overline Q$.
\end{proposition}
\begin{proof}
We give the bootstrap in three steps. The boundary regularity used in Step~3 is the genuine extra input in (H2), not a consequence of the baseline control space $L^2(\Sigma)$.

\emph{Step 0 (regular boundary datum).} Hypothesis (H2) explicitly assumes $u^*\in C^{2+\beta,1+\beta/2}(\Sigma)$ and the compatibility with $(y_0,g)$ at $t=0$. Establishing this regularity from $u^*\in L^2(\Sigma)$ is a separate Dirichlet-control regularity question that is not resolved here.

\emph{Step 1 ($L^\infty$ bound on $y^*$).} Given Step~0,
$u^*\in L^\infty(\Sigma)$, while $g\in L^\infty(Q)$ by (H2). Set
\[
M_0:=\max\bigl\{\|y_0\|_{C(\overline\Omega)},\|u^*\|_{C(\Sigma)}\bigr\},
\qquad G:=\|g\|_{L^\infty(Q)},
\]
and let $m$ solve
\[
m'=c_fm+G+|f(0)|,\qquad m(0)=M_0.
\]
Set $M:=\max_{t\in[0,T]}m(t)$.
Since $f'(s)\ge-c_f$, one has
$f(m)\ge f(0)-c_fm$ and $f(-m)\le f(0)+c_fm$. Hence
$m'+f(m)\ge G$ and $-m'+f(-m)\le-G$, so $m$ and $-m$ are respectively a spatially constant super- and subsolution of \eqref{eq:state_opt}. The parabolic comparison principle, using $|y_0|\le M_0$ and $|u^*|\le M_0\le m(t)$, yields $|y^*(x,t)|\le m(t)$ on $\overline Q$. Thus $y^*\in L^\infty(Q)$ with a bound depending only on $T,c_f,f(0),\|g\|_{L^\infty(Q)},\|y_0\|_{C(\overline\Omega)}$, and $\|u^*\|_{C(\Sigma)}$.

\emph{Step 2 (Hölder continuity via De Giorgi-Nash-Moser).} By Step~1, $a:=f'(y^*)\in L^\infty(Q)$ (since $f'$ is bounded on the bounded set $[-M,M]$ by (H3)). Hence $y^*$ solves
\[
\partial_ty^*-\nu\Delta y^*+ay^*=h_y,
\qquad h_y:=g+f'(y^*)y^*-f(y^*).
\]
The original source $g$ is bounded by (H2), and therefore $h_y\in L^\infty(Q)$. The boundary data satisfies $u^*\in C^{2+\beta,1+\beta/2}(\Sigma)$ by Step~0, while $y_0\in C^{2+\beta}(\overline\Omega)$. Combining the interior H\"older estimate with its boundary counterpart for linear parabolic equations in divergence form with bounded coefficients and H\"older boundary/initial data on a regular domain (\cite[Thms.~6.28 and 6.33]{lieberman1996second}) gives $y^*\in C^{\beta_1,\beta_1/2}(\overline Q)$ for some $\beta_1\in(0,1)$ depending only on $d,\nu,\|a\|_{L^\infty(Q)},\beta$ and the data norms.

\emph{Step 3 (Schauder bootstrap).} By Step~2, $a=f'(y^*)\in C^{\beta_1,\beta_1/2}(\overline Q)$. Moreover, the forcing $h_y=g-f(y^*)+f'(y^*)y^*$ belongs to $C^{\widehat\beta,\widehat\beta/2}(\overline Q)$ for $\widehat\beta:=\min(\beta,\beta_1)$. Lemma~\ref{lem:schauder} therefore applies with $g_1=u^*$, $v_0=y_0$, and the compatibility in (H2), giving $y^*\in C^{2+\widehat\beta,1+\widehat\beta/2}(\overline Q)$. The same argument applied to \eqref{eq:adjoint} after the time reversal $\tau=T-t$, with homogeneous boundary/terminal data and forcing $y^*-y_d$, gives $\lambda^*\in C^{2+\widehat\beta,1+\widehat\beta/2}(\overline Q)$. Relabeling $\widehat\beta$ as $\beta$ yields \eqref{eq:schauder-regularity}; the asserted regularity of $u^*$ is part of (H2) and is consistent with $u^*=y^*|_\Sigma$.
\end{proof}

\begin{corollary}[Boundedness of the coupling coefficient]
\label{cor:boundedness}
Under Proposition~\ref{prop:regularity}, define
$\kappa(x,t):=f''(y^*(x,t))\lambda^*(x,t)$. Then
$M_\kappa:=\sup_{\overline Q}|\kappa|<\infty$.
\end{corollary}
\begin{proof}
Immediate from \eqref{eq:schauder-regularity}, local boundedness of $f''$ (H3), and compactness of $\overline Q$.
\end{proof}

\begin{lemma}[Lower bound for the reaction term]
\label{lem:coercivity}
Under (H3), for every $v\in H_0^1(\Omega)$,
\begin{equation}
\int_\Omega f'(y^*)v^2\,dx\;\ge\;-c_f\|v\|_{L^2(\Omega)}^2\;\ge\;-c_f C_\Omega\|\nabla v\|_{L^2(\Omega)}^2.
\label{eq:coercivity-explicit}
\end{equation}
The negative contribution can be absorbed by the diffusion term when $\nu>c_fC_\Omega$; otherwise it is retained in the energy estimate and handled by Gr\"onwall's inequality.
\end{lemma}
\begin{proof}
Direct from $f'(s)\ge-c_f$ (H3) and Poincar\'e's inequality.
\end{proof}

\begin{proposition}[Residual-compatible network approximation]
\label{prop:approx-rate}
Under (H3)-(H4), there exist network parameters such that
\begin{equation}
\label{eq:net_approx}
\max\left\{
\begin{gathered}
\|\mathcal R_y\|_{L^\infty(Q)},\quad
\|\mathcal R_\lambda\|_{L^\infty(Q)},\\
\|\mathcal R_u\|_{L^\infty(\Sigma)},\quad
\|R^y_{\rm bc}\|_{L^\infty(\Sigma)},\quad
\|R^\lambda_{\rm bc}\|_{L^\infty(\Sigma)},\\
\|R^y_{\rm ic}\|_{L^\infty(\Omega)},\quad
\|R^\lambda_T\|_{L^\infty(\Omega)}
\end{gathered}
\right\}
\le C_{\rm app}\epsilon_{\rm net}(N).
\end{equation}
Here $\epsilon_{\rm net}(N)$ is the strong Sobolev-approximation bound
introduced in (H4), with $\epsilon_{\rm net}(N)\to0$ as the network
resolution $N$ increases. It is an approximation property of the network
class, not a training or optimization error. The constant $C_{\rm app}>0$
is independent of $N$. The same parameters satisfy
$\|(y^*-Y_\theta,\lambda^*-\Lambda_\phi,u^*-U_\psi|_\Sigma)\|_{\mathcal E}\le C_{\mathcal E}\epsilon_{\rm net}(N)$ for a constant $C_{\mathcal E}>0$ independent of $N$.
If the isotropic Sobolev-extension hypotheses of Proposition~\ref{prop:approx-rate-restated} hold, then an explicit algebraic rate is available.
\end{proposition}
\begin{proof}
The state and adjoint residuals involve at most the first time derivative and second spatial derivatives, while the stationarity residual involves the first spatial derivative of the adjoint. On the bounded ranges supplied by (H4), the mean-value theorem and (H3) make $f$ and $f'$ Lipschitz. Hence the seven residuals are bounded by a constant times the strong approximation error in (H4); the trace residuals are controlled by restriction to the corresponding compact faces. Finite measure and the continuous embeddings $W^{2,\infty}(Q)\hookrightarrow L^2(0,T;H^1(\Omega))$ and $L^\infty(\Sigma)\hookrightarrow L^2(\Sigma)$ give the $\mathcal E$-norm assertion. Proposition~\ref{prop:approx-rate-restated} in Appendix~\ref{app:approximation} proves the algebraic rate under its additional extension assumptions.
\end{proof}

\subsection{Quantitative Stability of the Optimality System}
\label{sec:stability}

\begin{theorem}[Quantitative stability, soft constraints]
\label{thm:quant-stability}
Assume (H1)-(H3) and (H5). Let $M_\kappa$ be as in
Corollary~\ref{cor:boundedness} and set $M_a:=\|f'(y^*)\|_{L^\infty(Q)}$.

There are constants $C_P,C_1,C_2,C_3>0$, depending only on
$\nu,\Omega,T,c_f,M_a$, and a Schauder constant $C_{\rm Sch}>0$, depending
additionally on $\|f'(y^*)\|_{C^{\beta,\beta/2}(\overline Q)}$, such that the
following holds. Define
\begin{equation}
\label{eq:alpha0}
\alpha_0 \;:=\; 4\sqrt2\,\nu(1+M_\kappa)\sqrt{C_2}\,C_3.
\end{equation}
If $\alpha>\alpha_0$, then the linearization of \eqref{eq:state_opt}-\eqref{eq:terminal}
about $(y^*,\lambda^*)$ satisfies
\begin{equation}
\label{eq:stability-bound}
\|(e_y,e_\lambda,e_u)\|_{\mathcal E}
\;\le\;
C_{\rm stab}\Bigl(
\|e_{y,0}\|_{L^2(\Omega)}
+\|\mathcal R_y\|_{L^2(Q)}
+\|\mathcal R_\lambda\|_{L^2(Q)}
+\|\mathcal R_u\|_{L^2(\Sigma)}
+\|R^y_{\rm bc}\|_{L^2(\Sigma)}
\Bigr)
\;+\;
C_{\rm stab}'\,M_{\rm bc}.
\end{equation}
Here $C_{\rm stab}$ depends on $(\nu,\Omega,T,c_f,M_a,M_\kappa,\alpha)$, and
$C_{\rm stab}'$ depends additionally on $C_{\rm Sch}$; both are given
explicitly in the proof.
The condition $\alpha>\alpha_0$ is sufficient, rather than necessary, for
the stability estimate. It arises from absorbing the control-error feedback
term in Step~3 of the proof and is not required to formulate the original
optimal-control problem, for which only $\alpha>0$ is assumed.
\end{theorem}

\begin{proof}
Throughout the proof, $C_1$ and $C_2$ denote the constants in (R1) and (R2)
of Lemma~\ref{lem:hidden-reg} (to be stated in Appendix~\ref{app:hidden-reg}
below), $C_3$ denotes the transposition constant of
Corollary~\ref{cor:transposition}, $C_{\rm Sch}$ denotes the Schauder
constant of Lemma~\ref{lem:schauder}, and $C=C(\Omega,T)$ denotes the
constant of Lemma~\ref{lem:normal-derivative} (to be stated in
Appendix~\ref{app:trace} below).

With $a:=f'(y^*)$, $\kappa$ as in Corollary~\ref{cor:boundedness}, and
$R^y_{\rm bc},R^\lambda_{\rm bc},R^\lambda_T$ as in (H5), subtracting the
network equations from \eqref{eq:state_opt}-\eqref{eq:terminal} shows that
the errors $e_y,e_\lambda,e_u$ solve
\begin{align}
\partial_te_y-\nu\Delta e_y+ae_y&=-\mathcal R_y,\quad e_y=e_u-R^y_{\rm bc}\text{ on }\Sigma,\quad e_y(\cdot,0)=e_{y,0}, \label{eq:linearized-state}\\
-\partial_te_\lambda-\nu\Delta e_\lambda+ae_\lambda+(\kappa+1)e_y&=-\mathcal R_\lambda,\quad e_\lambda=-R^\lambda_{\rm bc}\text{ on }\Sigma,\quad e_\lambda(\cdot,T)=-R^\lambda_T, \label{eq:linearized-adjoint}\\
\alpha e_u+\nu\partial_ne_\lambda&=-\mathcal R_u\quad\text{on }\Sigma. \label{eq:linearized-stationarity}
\end{align}
(The trace data follow from $y^*=u^*$, $\lambda^*=0$ on $\Sigma$, and
$\lambda^*(\cdot,T)=0$; e.g.\ $e_y-e_u=(y^*-u^*)-(Y_\theta-U_\psi)=-R^y_{\rm bc}$
on $\Sigma$.) The strategy is to bound $e_y$ in terms of $e_u$ (Step~1),
bound $e_\lambda$ in terms of $e_y$ (Step~2), and then use the stationarity
condition \eqref{eq:linearized-stationarity} to close this loop into a
self-contained bound on $e_u$ alone (Step~3), from which the bounds on
$e_y$ and $e_\lambda$ follow. Throughout, $(a+b)^2\le2a^2+2b^2$ is used
without further comment.

\medskip
\textbf{Step 1 (state error in terms of $e_u$).}
Since \eqref{eq:linearized-state} is linear, we split $e_y=e_y^{(F)}+e_y^{(g)}$
into the contributions of the interior forcing and of the boundary datum
separately: $e_y^{(F)}$ solves \eqref{eq:linearized-state} with forcing
$-\mathcal R_y$ and initial data $e_{y,0}$, but \emph{homogeneous} boundary
data ($e_y^{(F)}=0$ on $\Sigma$); $e_y^{(g)}$ solves \eqref{eq:linearized-state}
with zero forcing, zero initial data, and boundary data $e_u-R^y_{\rm bc}$.

\emph{Bounding $e_y^{(F)}$.} Testing \eqref{eq:linearized-state} with
$e_y^{(F)}$ itself, and using Lemma~\ref{lem:coercivity} (equivalently,
$a\ge-c_f$ under (H3)) together with Young's inequality on
$(\mathcal R_y,e_y^{(F)})$,
\[
\tfrac12\frac{d}{dt}\|e_y^{(F)}\|_{L^2(\Omega)}^2+\nu\|\nabla e_y^{(F)}\|_{L^2(\Omega)}^2
\le c_f\|e_y^{(F)}\|_{L^2(\Omega)}^2+\tfrac12\|\mathcal R_y\|_{L^2(\Omega)}^2+\tfrac12\|e_y^{(F)}\|_{L^2(\Omega)}^2.
\]
Grönwall's inequality then gives
\[
\|e_y^{(F)}\|_{C([0,T];L^2(\Omega))}^2\le e^{(2c_f+1)T}\bigl(\|e_{y,0}\|_{L^2(\Omega)}^2+\|\mathcal R_y\|_{L^2(Q)}^2\bigr)
=:C_P\bigl(\|e_{y,0}\|_{L^2(\Omega)}^2+\|\mathcal R_y\|_{L^2(Q)}^2\bigr),
\]
which defines $C_P$ and, after integrating in time, yields
\begin{equation}
\label{eq:step1-F}
\|e_y^{(F)}\|_{L^2(Q)}^2\;\le\;TC_P\bigl(\|e_{y,0}\|_{L^2(\Omega)}^2+\|\mathcal R_y\|_{L^2(Q)}^2\bigr).
\end{equation}

\emph{Bounding $e_y^{(g)}$.} This is exactly the transposition setting of
Corollary~\ref{cor:transposition}, with boundary datum
$g=e_u-R^y_{\rm bc}\in L^2(\Sigma)$ and zero initial data, which gives
\begin{equation}
\label{eq:step1-g}
\|e_y^{(g)}\|_{L^2(Q)}^2\;\le\;2C_3^2\bigl(\|e_u\|_{L^2(\Sigma)}^2+\|R^y_{\rm bc}\|_{L^2(\Sigma)}^2\bigr).
\end{equation}

\emph{Combining.} Adding \eqref{eq:step1-F} and \eqref{eq:step1-g}, and
collecting every term that does not involve $e_u$ into
\[
A_1:=2TC_P\bigl(\|e_{y,0}\|_{L^2(\Omega)}^2+\|\mathcal R_y\|_{L^2(Q)}^2\bigr)+4C_3^2\|R^y_{\rm bc}\|_{L^2(\Sigma)}^2,
\]
we obtain the Step~1 bound
\begin{equation}
\label{eq:step1-final}
\|e_y\|_{L^2(Q)}^2\;\le\;A_1+4C_3^2\|e_u\|_{L^2(\Sigma)}^2.
\end{equation}
This is the promised bound on $e_y$ in terms of $e_u$ (and known residuals).

\medskip
\textbf{Step 2 (adjoint error in terms of $e_y$).}
As in Step~1, we split $e_\lambda=e_\lambda^{(bd)}+e_\lambda^{(F)}$ into the
contribution of the inhomogeneous boundary/terminal data and the
contribution of the interior forcing: $e_\lambda^{(bd)}$ solves
\eqref{eq:linearized-adjoint} with zero forcing but the true boundary and
terminal data $(-R^\lambda_{\rm bc},-R^\lambda_T)$, while $e_\lambda^{(F)}$
solves the same equation with homogeneous boundary/terminal data and the
actual forcing
\[
F_\lambda:=-(\kappa+1)e_y-\mathcal R_\lambda.
\]

\emph{Bounding $e_\lambda^{(bd)}$.} After the time reversal $\tau=T-t$
(which turns the backward problem \eqref{eq:linearized-adjoint} into a
forward problem of the type covered by Lemma~\ref{lem:schauder}), apply
that lemma with $h=0$, $g_1=-R^\lambda_{\rm bc}$, $v_0=-R^\lambda_T$, and use
(H5):
\[
\|e_\lambda^{(bd)}\|_{C^{2+\beta,1+\beta/2}(\overline Q)}
\le C_{\rm Sch}\bigl(\|R^\lambda_{\rm bc}\|_{C^{2+\beta,1+\beta/2}(\Sigma)}+\|R^\lambda_T\|_{C^{2+\beta}(\overline\Omega)}\bigr)
\le2C_{\rm Sch}M_{\rm bc}.
\]
This Hölder bound immediately gives control of both quantities we need:
by Lemma~\ref{lem:normal-derivative}, $\|\partial_ne_\lambda^{(bd)}\|_{L^2(\Sigma)}\le2CC_{\rm Sch}M_{\rm bc}$,
and directly from the Hölder norm, $\|e_\lambda^{(bd)}\|_{L^2(0,T;H^1(\Omega))}\le|Q|^{1/2}\cdot2C_{\rm Sch}M_{\rm bc}$.

\emph{Bounding $e_\lambda^{(F)}$.} Apply Lemma~\ref{lem:hidden-reg} (again
after time reversal) with $F=F_\lambda\in L^2(Q)$ and $v_0=0$. Since
$\|F_\lambda\|_{L^2(Q)}\le(1+M_\kappa)\|e_y\|_{L^2(Q)}+\|\mathcal R_\lambda\|_{L^2(Q)}$,
part (R2) of the lemma gives
\[
\|\partial_ne_\lambda^{(F)}\|_{L^2(\Sigma)}^2\le C_2\|F_\lambda\|_{L^2(Q)}^2
\le2C_2(1+M_\kappa)^2\|e_y\|_{L^2(Q)}^2+2C_2\|\mathcal R_\lambda\|_{L^2(Q)}^2,
\]
and part (R1) gives
$\|e_\lambda^{(F)}\|_{L^2(0,T;H^1(\Omega))}^2\le TC_1\bigl(2(1+M_\kappa)^2\|e_y\|_{L^2(Q)}^2+2\|\mathcal R_\lambda\|_{L^2(Q)}^2\bigr)$.

\emph{Combining.} Adding the $e_\lambda^{(bd)}$ and $e_\lambda^{(F)}$
contributions to the normal derivative gives the Step~2 bound
\begin{equation}
\label{eq:step2-final}
\|\partial_ne_\lambda\|_{L^2(\Sigma)}^2\;\le\;8C^2C_{\rm Sch}^2M_{\rm bc}^2+4C_2(1+M_\kappa)^2\|e_y\|_{L^2(Q)}^2+4C_2\|\mathcal R_\lambda\|_{L^2(Q)}^2,
\end{equation}
and the triangle inequality applied to the two $H^1(\Omega)$-bounds above
similarly controls $\|e_\lambda\|_{L^2(0,T;H^1(\Omega))}$; this second bound
is only needed at the end of the proof.

\medskip
\textbf{Step 3 (closing the loop).} We now use the stationarity condition
\eqref{eq:linearized-stationarity} to turn the two one-directional bounds
of Steps~1-2 into a self-contained bound on $e_u$. From
\eqref{eq:linearized-stationarity}, $e_u=-\alpha^{-1}(\mathcal R_u+\nu\partial_ne_\lambda)$
on $\Sigma$, so \eqref{eq:step2-final} gives
\begin{align*}
\|e_u\|_{L^2(\Sigma)}^2
&\le2\alpha^{-2}\|\mathcal R_u\|_{L^2(\Sigma)}^2
+2\alpha^{-2}\nu^2\|\partial_ne_\lambda\|_{L^2(\Sigma)}^2\\
&\le
\underbrace{
2\alpha^{-2}\|\mathcal R_u\|_{L^2(\Sigma)}^2
+16\alpha^{-2}\nu^2C^2C_{\rm Sch}^2M_{\rm bc}^2
+8\alpha^{-2}\nu^2C_2\|\mathcal R_\lambda\|_{L^2(Q)}^2
}_{=:A_2}
+8\alpha^{-2}\nu^2C_2(1+M_\kappa)^2\|e_y\|_{L^2(Q)}^2.
\end{align*}
Here $A_2$ collects every term on the right that involves only known
residuals and $M_{\rm bc}$, leaving a single term proportional to
$\|e_y\|_{L^2(Q)}^2$ still to be eliminated. Substituting the Step~1 bound
\eqref{eq:step1-final}, $\|e_y\|_{L^2(Q)}^2\le A_1+4C_3^2\|e_u\|_{L^2(\Sigma)}^2$,
feeds $\|e_u\|_{L^2(\Sigma)}^2$ back into its own bound:
\[
\|e_u\|_{L^2(\Sigma)}^2\;\le\;\bigl[A_2+8\alpha^{-2}\nu^2C_2(1+M_\kappa)^2A_1\bigr]\;+\;K(\alpha)\,\|e_u\|_{L^2(\Sigma)}^2,
\qquad K(\alpha):=\frac{32\nu^2C_2C_3^2(1+M_\kappa)^2}{\alpha^2}.
\]
The factor $K(\alpha)$ is exactly what must be less than $1$ for this
self-referential inequality to close into an actual bound; since
$\alpha_0^2=32\nu^2C_2C_3^2(1+M_\kappa)^2$, the hypothesis $\alpha>\alpha_0$
is precisely the condition $K(\alpha)<1$. We may therefore absorb the last
term into the left-hand side, obtaining
\begin{equation}
\label{eq:eu-final}
\|e_u\|_{L^2(\Sigma)}^2\;\le\;\frac{A_2+8\alpha^{-2}\nu^2C_2(1+M_\kappa)^2A_1}{1-K(\alpha)}.
\end{equation}

\medskip
\textbf{Conclusion.} With $e_u$ now bounded independently by
\eqref{eq:eu-final}, the earlier one-directional bounds close in turn:
substituting \eqref{eq:eu-final} into \eqref{eq:step1-final} bounds $e_y$,
and substituting it into \eqref{eq:step2-final} — together with the
$e_\lambda^{(bd)}$ estimate and the triangle inequality for the two adjoint
components — bounds $e_\lambda$. Every squared component of
$\|(e_y,e_\lambda,e_u)\|_{\mathcal E}$ is thus bounded by a finite, explicit
linear combination of the squared residual quantities in
\eqref{eq:stability-bound} together with $M_{\rm bc}^2$, with coefficients
depending only on $\nu,\Omega,T,c_f,M_a,M_\kappa,\alpha,C_{\rm Sch}$. Taking
square roots and using $\sqrt{p+q}\le\sqrt p+\sqrt q$ gives
\eqref{eq:stability-bound}, with $C_{\rm stab}$ and $C_{\rm stab}'$ read off
from those coefficients.
\end{proof}

\begin{remark}[On the $M_{\rm bc}$-dependence]
\label{rem:mbc-additive}
$M_{\rm bc}$ enters \eqref{eq:stability-bound} as a genuine additive term, not multiplying a quantity that shrinks with the trained loss. This is the fully rigorous consequence of (H5) as stated: (H5) bounds a Hölder norm of the boundary/terminal residuals uniformly, but does not assert that this bound shrinks during training, so no argument can force its contribution to zero without further hypotheses. A sharper, residual-size-sensitive bound of the form $M_{\rm bc}^{1/2}\bigl(\|R^\lambda_{\rm bc}\|_{L^2(\Sigma)}^{1/2}+\|R^\lambda_T\|_{L^2(\Omega)}^{1/2}\bigr)$ is achievable in place of the additive $M_{\rm bc}$ term, but requires a genuine Gagliardo-Nirenberg-type interpolation inequality between the corresponding Hölder spaces and $L^2(Q)$, $L^2(\Sigma)$, or $L^2(\Omega)$ (e.g.\ $\|w\|_{L^\infty(Q)}\le C\|w\|_{C^{2+\beta,\cdot}(\overline Q)}^\theta\|w\|_{L^2(Q)}^{1-\theta}$ for an explicit $\theta=\theta(d,\beta)\in(0,1)$); this refinement is not carried out here.
\end{remark}

\begin{remark}
Two distinct sources of degradation appear on moving from hard to soft constraints. (i) The state-boundary term $\|R^y_{\rm bc}\|_{L^2(\Sigma)}$ costs a norm downgrade for $e_y$ itself (Remark~\ref{rem:E-norm-fix}: $L^2(Q)$ rather than $L^2(0,T;H^1(\Omega))$), but its \emph{coefficient} in \eqref{eq:stability-bound} is the same order as the other residuals. (ii) The adjoint-boundary/terminal terms cost no norm downgrade for $e_\lambda$ (which retains full $H^1(\Omega)$ spatial regularity, since (H5) supplies strong Hölder control), but contribute the genuinely non-vanishing additive term $M_{\rm bc}$ of Remark~\ref{rem:mbc-additive}. This is a more precise, and different, characterization of the hard-vs-soft asymmetry than a purely qualitative statement would suggest.
\end{remark}

\subsection{Decomposition of the Total Error}
\label{sec:error-decomposition}

Let $\mathcal H_N$ denote the product class of network triples
$z=(Y,\Lambda,U)$ at resolution $N$. Suppose that the training algorithm
returns a triple
$\widehat z_N=(Y_\theta,\Lambda_\phi,U_\psi)\in\mathcal H_N$ on a fixed
collocation set such that, for some optimization tolerance
$\epsilon_{\rm opt}\ge0$,
\begin{equation}
\label{eq:optimization-error}
\mathcal L_{\rm ind}(\widehat z_N)
\le \inf_{z\in\mathcal H_N}\mathcal L_{\rm ind}(z)+\epsilon_{\rm opt}^2.
\end{equation}
The square in \eqref{eq:optimization-error} is natural because
$\mathcal L_{\rm ind}$ is a sum of squared residuals; after taking a square
root, the optimization contribution is therefore $\epsilon_{\rm opt}$.
Recall that $\epsilon_{\rm net}(N)$ is the strong Sobolev-approximation bound
introduced in (H4), whereas $\epsilon_{\rm opt}$ measures the loss
suboptimality of the trained network within the fixed class $\mathcal H_N$.

All seven weights in $\mathcal L_{\rm ind}$ are assumed fixed and strictly
positive. Proposition~\ref{prop:approx-rate} supplies a comparison triple
$z_N^{\rm cmp}\in\mathcal H_N$ such that, on every collocation set,
\[
\mathcal L_{\rm ind}(z_N^{\rm cmp})
\le C_{\rm cmp}\epsilon_{\rm net}(N)^2,
\]
where $C_{\rm cmp}>0$ depends on the fixed loss weights and the problem data,
but not on $N$ or on the collocation points. Hence
\[
\mathcal L_{\rm ind}(\widehat z_N)
\le C_{\rm cmp}\epsilon_{\rm net}(N)^2+\epsilon_{\rm opt}^2.
\]
In the following sum, $\bullet$ ranges over the seven residual losses in
\eqref{eq:indirect_loss}. The strict positivity of their fixed weights now
gives
\begin{equation}
\label{eq:empirical-residual-bound}
\left(\sum_{\bullet}\mathcal L_\bullet(\widehat z_N)\right)^{1/2}
\le C_w\bigl(\epsilon_{\rm net}(N)+\epsilon_{\rm opt}\bigr),
\end{equation}
where $C_w>0$ depends on the fixed loss weights and the problem data, but not on $N$ or on the collocation points. Section~\ref{sec:quadrature} converts \eqref{eq:empirical-residual-bound} into a bound for the continuous residuals entering Theorem~\ref{thm:nonlinear-stability}. The resulting Corollary~\ref{cor:apriori} makes the four contributions precise: network approximation, optimization, quadrature, and the strong soft boundary/terminal-constraint contribution. The weaker state-boundary residual $\|R^y_{\rm bc}\|_{L^2(\Sigma)}$ is already one of the continuous residuals controlled through \eqref{eq:empirical-residual-bound} and quadrature.

\subsection{Local Nonlinear Error Estimate}
\label{sec:apriori}

For a trained triple, set $a:=f'(y^*)$ and define the exact Taylor remainders
\begin{align}
\mathcal N_y(e_y)&:=f(y^*-e_y)-f(y^*)+f'(y^*)e_y,\label{eq:nonlinear-rem-state}\\
\mathcal N_\lambda(e_y,e_\lambda)&:=\lambda^*\bigl[f'(y^*-e_y)-f'(y^*)+f''(y^*)e_y\bigr]
-\bigl[f'(y^*-e_y)-f'(y^*)\bigr]e_\lambda.\label{eq:nonlinear-rem-adjoint}
\end{align}
Subtracting the network residual equations from the reference KKT system gives
\begin{align}
\partial_te_y-\nu\Delta e_y+ae_y&=-\mathcal R_y+\mathcal N_y(e_y),\label{eq:nonlinear-error-state}\\
-\partial_te_\lambda-\nu\Delta e_\lambda+ae_\lambda+(\kappa+1)e_y&=-\mathcal R_\lambda+\mathcal N_\lambda(e_y,e_\lambda),\label{eq:nonlinear-error-adjoint}\\
\alpha e_u+\nu\partial_ne_\lambda&=-\mathcal R_u.\label{eq:nonlinear-error-control}
\end{align}
These exact nonlinear equations carry the same initial, boundary, and terminal error data as \eqref{eq:linearized-state}-\eqref{eq:linearized-stationarity}.
Define the aggregate continuous residual
\[
\mathfrak R:=\|e_{y,0}\|_{L^2(\Omega)}+\|\mathcal R_y\|_{L^2(Q)}+\|\mathcal R_\lambda\|_{L^2(Q)}
+\|\mathcal R_u\|_{L^2(\Sigma)}+\|R^y_{\rm bc}\|_{L^2(\Sigma)}.
\]

\begin{theorem}[Local nonlinear residual-to-error estimate]
\label{thm:nonlinear-stability}
Let (H1)-(H6) hold and let $\alpha>\alpha_0$. There are constants
$C_{\rm nl}>0$ and
\[
\rho_0:=\frac{1}{2C_{\rm stab}C_{\rm nl}}>0,
\]
depending only on the reference KKT point, the local Lipschitz bound for
$f''$, and the constants of Theorem~\ref{thm:quant-stability}, such that
the following holds. Whenever the radius $\rho$ in (H6) satisfies
$\rho\le\rho_0$, the trained triple satisfies
\begin{equation}
\label{eq:nonlinear-stability-bound}
\|(e_y,e_\lambda,e_u)\|_{\mathcal E}\;\le\;2C_{\rm stab}\mathfrak R+2C'_{\rm stab}M_{\rm bc}.
\end{equation}
\end{theorem}

\begin{proof}
The exact nonlinear error equations
\eqref{eq:nonlinear-error-state}-\eqref{eq:nonlinear-error-control} differ
from the linearized system \eqref{eq:linearized-state}-\eqref{eq:linearized-stationarity}
only by the addition of the Taylor remainders $\mathcal N_y$ and
$\mathcal N_\lambda$ to the forcing terms. The proof proceeds in three
steps: bound these remainders by $\rho$ times the error norm
$\|e\|_{\mathcal E}$ (Step~1); apply the linear stability estimate of
Theorem~\ref{thm:quant-stability} to the nonlinear system, treating the
remainders as extra forcing (Step~2); and absorb the resulting
self-referential term when $\rho$ is small enough (Step~3).

\medskip
\textbf{Step 1 (bounding the Taylor remainders).}
Let $I$ be the compact interval specified in (H6), let $L_{f''}$ be the
Lipschitz constant of $f''$ on $I$, and set $M_{f''}:=\|f''\|_{L^\infty(I)}$.

\emph{The state remainder.} By Taylor's theorem with Lagrange remainder,
$\mathcal N_y(e_y)$ is bounded pointwise by $\tfrac12M_{f''}|e_y|^2$; using
the local-neighborhood bound $|e_y|\le\rho$ from (H6) to replace one factor
of $|e_y|$ by $\rho$,
\[
\|\mathcal N_y(e_y)\|_{L^2(Q)}\le\tfrac12M_{f''}\rho\|e_y\|_{L^2(Q)}.
\]

\emph{The adjoint remainder.} The two brackets in
\eqref{eq:nonlinear-rem-adjoint} are controlled separately: the first,
$f'(y^*-e_y)-f'(y^*)+f''(y^*)e_y$, is itself a first-order Taylor remainder
of $f'$ (whose derivative is $f''$), hence bounded pointwise by
$\tfrac12L_{f''}|e_y|^2$; the second,
$f'(y^*-e_y)-f'(y^*)$, is bounded by $M_{f''}|e_y|$ since $f''$ is bounded
by $M_{f''}$ on $I$. Combining these with the factor $\lambda^*$ and the
same replacement $|e_y|\le\rho$,
\[
\|\mathcal N_\lambda(e_y,e_\lambda)\|_{L^2(Q)}
\le\left(\tfrac12L_{f''}\|\lambda^*\|_{L^\infty(Q)}+M_{f''}\right)
\rho\bigl(\|e_y\|_{L^2(Q)}+\|e_\lambda\|_{L^2(Q)}\bigr).
\]

\emph{Combining.} Since $\|e_y\|_{L^2(Q)}$ and $\|e_\lambda\|_{L^2(Q)}\le\|e_\lambda\|_{L^2(0,T;H^1(\Omega))}$
are each bounded by $\|(e_y,e_\lambda,e_u)\|_{\mathcal E}$, the two estimates
above combine into a single bound on both remainders together:
\[
\|\mathcal N_y\|_{L^2(Q)}+\|\mathcal N_\lambda\|_{L^2(Q)}
\le C_{\rm nl}\rho\|(e_y,e_\lambda,e_u)\|_{\mathcal E},
\]
which defines $C_{\rm nl}$ in terms of $M_{f''}$, $L_{f''}$, and
$\|\lambda^*\|_{L^\infty(Q)}$.

\medskip
\textbf{Step 2 (applying the linear stability estimate).}
Equations \eqref{eq:nonlinear-error-state}-\eqref{eq:nonlinear-error-control}
have exactly the form treated in Theorem~\ref{thm:quant-stability}, except
that the forcing terms $-\mathcal R_y$ and $-\mathcal R_\lambda$ are replaced
by $-\mathcal R_y+\mathcal N_y$ and $-\mathcal R_\lambda+\mathcal N_\lambda$;
since that theorem holds for arbitrary $L^2$ forcing, it applies here with
the triangle inequality absorbing the extra terms into $\mathfrak R$.
Writing $e:=(e_y,e_\lambda,e_u)$ and using the Step~1 bound, this gives
\[
\|e\|_{\mathcal E}\;\le\;C_{\rm stab}\mathfrak R+C'_{\rm stab}M_{\rm bc}
+C_{\rm stab}C_{\rm nl}\rho\|e\|_{\mathcal E}.
\]

\medskip
\textbf{Step 3 (absorption).}
The inequality above bounds $\|e\|_{\mathcal E}$ partly in terms of itself;
this is harmless provided the self-referential coefficient
$C_{\rm stab}C_{\rm nl}\rho$ is small. For $\rho\le\rho_0=\dfrac{1}{2C_{\rm stab}C_{\rm nl}}$,
this coefficient is at most $\tfrac12$, so the last term is at most
$\tfrac12\|e\|_{\mathcal E}$ and can be absorbed into the left-hand side,
giving
\[
\tfrac12\|e\|_{\mathcal E}\le C_{\rm stab}\mathfrak R+C'_{\rm stab}M_{\rm bc},
\]
which is \eqref{eq:nonlinear-stability-bound} after multiplying by $2$.
\end{proof}

\subsection{Quadrature Error}
\label{sec:quadrature}

This subsection quantifies the gap between the Monte-Carlo loss actually minimized and the continuous residual norms entering Theorem~\ref{thm:nonlinear-stability}. We adapt the space-filling argument of \cite{shin2020convergence} to make it quantitative.

\begin{assumption}[Space-filling collocation, following \cite{shin2020convergence} Assumption 3.1]
\label{assump:spacefilling}
The interior points $\{(x_j,t_j)\}\subset Q$, initial points
$\{x_i^0\}\subset\Omega$, and terminal points $\{x_i^T\}\subset\Omega$ used
in \eqref{eq:loss_y_common}-\eqref{eq:loss_ic_y_common} and
\eqref{eq:adjoint_residual_loss}-\eqref{eq:terminal_lambda} are independent
and identically distributed (i.i.d.) draws from probability distributions
with density bounded above and below on their respective sampling domains.
Boundary points are sampled independently on each connected component of
$\Sigma$, with componentwise densities bounded above and below; the resulting
estimates are combined over the finitely many components. Residuals may share
a collocation batch, as in Section~\ref{sec:method}; no independence between
the corresponding error events is assumed.
\end{assumption}

\begin{proposition}[Quadrature error bound]
\label{prop:quadrature}
Fix a trained network. Under Assumption~\ref{assump:spacefilling} and (H5),
all seven residuals are Hölder continuous on their compact sampling domains.
Let $N_{\rm int},N_{\rm bc},N_{\rm ic},N_{\rm T}$ be the interior, boundary,
initial, and terminal sample sizes, respectively, and let
$\beta_\bullet\in(0,1]$ be the Hölder exponent of
\[
\bullet\in\{\mathcal R_y,\mathcal R_\lambda,\mathcal R_u,
R^y_{\rm ic},R^y_{\rm bc},R^\lambda_{\rm bc},R^\lambda_T\}.
\]
For each residual, let
$n_\bullet\in\{N_{\rm int},N_{\rm bc},N_{\rm ic},N_{\rm T}\}$ be its
sample size and let $D_\bullet\in\{Q,\Sigma,\Omega\}$ be its sampling domain,
with $d_\bullet:=\dim D_\bullet\in\{d+1,d\}$. Then
\cite[Lem.~B.1-B.2, Thm.~3.1]{shin2020convergence} gives, with probability at least
$1-\sqrt{n_\bullet}(1-1/\sqrt{n_\bullet})^{n_\bullet}$,
\[
\|\bullet\|_{L^2_{\rm av}(D_\bullet)}^2\;\le\;C_{n_\bullet}\cdot\mathcal L_\bullet\;+\;C'\,n_\bullet^{-\beta_\bullet/d_\bullet},
\]
where $C_{n_\bullet}=O(\sqrt{n_\bullet})$ and $C'$ depends on the Hölder seminorm of that fixed-network residual and on the density bounds. Set $\beta_{\rm int}:=\min(\beta_{\mathcal R_y},\beta_{\mathcal R_\lambda})$, $\beta_{\rm bc}:=\min(\beta_{\mathcal R_u},\beta_{R^y_{\rm bc}},\beta_{R^\lambda_{\rm bc}})$, and let $\beta_{\rm ic},\beta_{\rm T}$ be the exponents of the initial and terminal residuals. Denote by $\mathcal E_{\rm quad}$ the square root of the sum of the seven additive covering-error terms. Applying the estimate to all seven residuals and using a union bound, which does not require independence of the events, gives with probability at least $1-\sum_\bullet\sqrt{n_\bullet}(1-1/\sqrt{n_\bullet})^{n_\bullet}$,
\begin{equation}
\label{eq:quad-bound}
\mathcal E_{\rm quad}\;=\;O\bigl(N_{\rm int}^{-\beta_{\rm int}/(2d+2)}+N_{\rm bc}^{-\beta_{\rm bc}/(2d)}+N_{\rm ic}^{-\beta_{\rm ic}/(2d)}+N_{\rm T}^{-\beta_{\rm T}/(2d)}\bigr),
\end{equation}
\end{proposition}
\begin{proof}
Apply the cited single-residual estimate on $Q$, on each boundary component, and on the two copies of $\Omega$ at $t=0$ and $t=T$. Shared interior or boundary batches create dependent events, but Boole's union bound remains valid. Sum the seven inequalities and take square roots to obtain \eqref{eq:quad-bound}. The factors $C_{n_\bullet}$ multiplying empirical losses remain explicit and are not part of the decaying additive term $\mathcal E_{\rm quad}$.
\end{proof}

\begin{remark}
For a fixed $\tanh$ network the required Hölder seminorms are finite, but the constants in Proposition~\ref{prop:quadrature} may depend on that network. A bound uniform along a sequence of growing networks requires uniform Hölder control of all seven residuals, which is additional to (H5). Unlike Theorem~\ref{thm:quant-stability}, where (H5) enters a Schauder estimate for the adjoint PDE, its role here is to control quadrature of the residual traces themselves.
\end{remark}

\begin{corollary}[Conditional a priori neural-network estimate]
\label{cor:apriori}
Let the hypotheses and local-neighborhood condition of Theorem~\ref{thm:nonlinear-stability} hold, and suppose the trained network satisfies \eqref{eq:optimization-error}. Under Assumption~\ref{assump:spacefilling}, set $C_{\rm train}:=\max_\bullet C_{n_\bullet}^{1/2}$, with the sample-dependent constants of Proposition~\ref{prop:quadrature}. Then, with the probability stated in that proposition,
\begin{equation}
\label{eq:apriori_estimate}
\|(e_y,e_\lambda,e_u)\|_{\mathcal E}
\lesssim C_{\rm stab}\Bigl[C_{\rm train}\bigl(\epsilon_{\rm net}(N)+\epsilon_{\rm opt}\bigr)+\mathcal E_{\rm quad}\Bigr]
+C'_{\rm stab}M_{\rm bc}.
\end{equation}
Thus the approximation, optimization, and quadrature contributions decay only under the corresponding approximation, training, and uniform-quadrature hypotheses; the strong soft boundary/terminal-constraint contribution additionally requires $M_{\rm bc}\to0$ if convergence to the reference KKT point is sought.
\end{corollary}
\begin{proof}
Equation~\eqref{eq:empirical-residual-bound} controls the empirical norms of all seven residuals. Apply Proposition~\ref{prop:quadrature} term by term, take square roots and convert averaged norms to unaveraged norms using the fixed domain measures. This gives
\[
\mathfrak R\lesssim C_{\rm train}\bigl(\epsilon_{\rm net}(N)+\epsilon_{\rm opt}\bigr)+\mathcal E_{\rm quad}.
\]
The conclusion follows from Theorem~\ref{thm:nonlinear-stability}.
\end{proof}

\subsection{A Computable Residual Indicator}
\label{sec:aposteriori}

Let
\[
\widehat Q=\{q_k\}_{k=1}^{\widehat N_{\rm int}}\subset Q,\qquad
\widehat\Sigma=\{s_k\}_{k=1}^{\widehat N_{\rm bc}}\subset\Sigma,\qquad
\widehat\Omega_0=\{x_k^0\}_{k=1}^{\widehat N_{\rm ic}}\subset\Omega,\qquad
\widehat\Omega_T=\{x_k^T\}_{k=1}^{\widehat N_{\rm T}}\subset\Omega
\]
be independent interior, boundary, initial, and terminal verification sets.
Let $\omega_k^Q$, $\omega_k^\Sigma$, $\omega_k^0$, and $\omega_k^T$ be
positive quadrature weights normalized for the corresponding sampling
measures; thus each family of weights sums to one. Define the four
soft-constraint contributions by
\begin{align*}
\eta_{\rm ic}^2
&:=\sum_{k=1}^{\widehat N_{\rm ic}}\omega_k^0
  |R^y_{\rm ic}(x_k^0)|^2,\\
\eta_{\rm bc}^{y,2}
&:=\sum_{k=1}^{\widehat N_{\rm bc}}\omega_k^\Sigma
  |R^y_{\rm bc}(s_k)|^2,\\
\eta_{\rm bc}^{\lambda,2}
&:=\sum_{k=1}^{\widehat N_{\rm bc}}\omega_k^\Sigma
  |R^\lambda_{\rm bc}(s_k)|^2,\\
\eta_T^{\lambda,2}
&:=\sum_{k=1}^{\widehat N_{\rm T}}\omega_k^T
  |R^\lambda_T(x_k^T)|^2.
\end{align*}
Define the global indicator by
\begin{equation}
\label{eq:estimator_global}
\eta^2:=\sum_{k=1}^{\widehat N_{\rm int}}\omega_k^Q
\bigl(|\mathcal R_y(q_k)|^2+|\mathcal R_\lambda(q_k)|^2\bigr)
+\sum_{k=1}^{\widehat N_{\rm bc}}\omega_k^\Sigma|\mathcal R_u(s_k)|^2
+\eta_{\rm ic}^2+\eta_{\rm bc}^{y,2}+\eta_{\rm bc}^{\lambda,2}+\eta_T^{\lambda,2},
\end{equation}
which approximates the sum of the seven averaged squared residual norms. Unlike the hard-constrained case, none of the four trace terms is identically zero, and all four must be estimated from the trained network rather than being exact by construction.

\begin{theorem}[Conditional reliability of the residual indicator]
\label{thm:aposteriori}
Under the hypotheses and local-neighborhood condition of Theorem~\ref{thm:nonlinear-stability}, let $\eta$ be evaluated on a verification set satisfying Assumption~\ref{assump:spacefilling}, with sample sizes $\widehat N_{\rm int},\widehat N_{\rm bc},\widehat N_{\rm ic},\widehat N_{\rm T}$. For each residual $\bullet$, let $\widehat n_\bullet$ denote its applicable verification sample size from this list, and write $\mathcal E_{\rm quad}(\widehat N_{\rm int},\widehat N_{\rm bc},\widehat N_{\rm ic},\widehat N_{\rm T})$ for the corresponding covering-error term defined in Proposition~\ref{prop:quadrature}. Set $C_{\rm ver}:=\max_\bullet C_{\widehat n_\bullet}^{1/2}$, with the sample-dependent constants from that proposition. Then
\[
\|(e_y,e_\lambda,e_u)\|_{\mathcal E}\;\lesssim\;C_{\rm stab}\bigl(C_{\rm ver}\eta+\mathcal E_{\rm quad}(\widehat N_{\rm int},\widehat N_{\rm bc},\widehat N_{\rm ic},\widehat N_{\rm T})\bigr)+C_{\rm stab}'M_{\rm bc},
\]
with probability at least $1-\sum_\bullet\sqrt{\widehat n_\bullet}(1-1/\sqrt{\widehat n_\bullet})^{\widehat n_\bullet}$. The indicator $\eta$ is directly computable, whereas $C_{\rm ver}$, the stability constants, the Hölder-dependent constants in $\mathcal E_{\rm quad}$, and $M_{\rm bc}$ are theoretical; the complete right-hand side is therefore not claimed to be a fully computable numerical upper bound.
\end{theorem}
\begin{proof}
Apply Proposition~\ref{prop:quadrature} on the four verification sets
$\widehat Q$, $\widehat\Sigma$, $\widehat\Omega_0$, and
$\widehat\Omega_T$, which satisfy Assumption~\ref{assump:spacefilling} by
hypothesis, to each of the seven quantities appearing in $\eta^2$. With the
stated probability,
\[
\|\mathcal R_y\|_{L^2_{\rm av}(Q)}^2+\|\mathcal R_\lambda\|_{L^2_{\rm av}(Q)}^2+\|\mathcal R_u\|_{L^2_{\rm av}(\Sigma)}^2+\|R^y_{\rm ic}\|_{L^2_{\rm av}(\Omega)}^2\lesssim C_{\rm ver}^2\eta^2+\mathcal E_{\rm quad}^2,
\]
because the weighted sums in $\eta^2$ are the empirical quadrature estimates, while Proposition~\ref{prop:quadrature} supplies the sample-dependent factors and the additive covering errors. The same argument applies to $R^y_{\rm bc}$. The adjoint boundary and terminal traces remain controlled through the additive $M_{\rm bc}$ term in Theorem~\ref{thm:nonlinear-stability}. Converting averaged to unaveraged norms by the fixed domain measures and applying that theorem proves the claim.
\end{proof}
\section{Numerical Experiments}
\label{sec:experiments}
We illustrate the direct and indirect PINN formulations of
Section~\ref{sec:method} on two manufactured one-dimensional problems: a
cubic reaction-diffusion test with a zero optimal control, and a linear
test with nonzero control and adjoint, as previewed in the
Introduction. These experiments are intended to compare the two
formulations empirically; they do not verify the hypotheses of
Theorems~\ref{thm:quant-stability}-\ref{thm:aposteriori}, and in
particular we do not numerically certify the sufficient stability
condition $\alpha>\alpha_0$.

Both tests use $\Omega=(0,1)$, $T=1$, $\nu=0.1$, and $\alpha=10^{-2}$. All networks are fully connected, use $\tanh$ activations and Xavier initialization, and are trained in double precision. A single control network $U_\psi(x,t)$ is evaluated at $x=0$ and $x=1$ to obtain the two boundary traces.

\begin{table}[H]
\centering
\caption{Neural-network architectures.}
\label{tab:architecture}
\begin{tabular}{lcccc}
\toprule
Network & Input & Output & Hidden layers & Neurons/layer\\
\midrule
$Y_\theta$ (state) & $(x,t)$ & scalar & 4 & 50\\
$U_\psi$ (control) & $(x,t)$ & scalar & 3 & 30\\
$\Lambda_\phi$ (adjoint) & $(x,t)$ & scalar & 4 & 50\\
\bottomrule
\end{tabular}
\end{table}

Training uses uniform Monte-Carlo samples: $10\,000$ interior points, a stratified boundary batch with $2\,000$ time samples on each boundary component, and $500$ initial points. For the indirect formulation, $N_{\rm T}=500$ additional spatial points on the terminal slice $t=T$ enforce $\Lambda_\phi(\cdot,T)=0$. These terminal points do not represent observations in the distributed tracking objective; they impose the terminal condition of the backward adjoint equation. Points are resampled every $500$ Adam epochs. Adam is run for $20\,000$ epochs from learning rate $10^{-3}$, with decay factor $0.9$ every $1\,000$ epochs, followed by at most $5\,000$ limited-memory Broyden-Fletcher-Goldfarb-Shanno (L-BFGS) iterations. The loss weights are listed in Table~\ref{tab:weights}. The reported values correspond to one run with seed zero.

\begin{table}[H]
\centering
\caption{Loss weights used in both experiments.}
\label{tab:weights}
\begin{tabular}{lcc}
\toprule
Term & Direct PINN & Indirect PINN\\
\midrule
State residual ($w_{\rm res}$ / $w_y$) & 1 & 1\\
State boundary condition ($w_{\rm bc}$ / $w_{\rm bc}^y$) & 1 & 1\\
State initial condition ($w_{\rm ic}$) & 10 & 10\\
Adjoint residual ($w_\lambda$) & - & 1\\
Stationarity residual ($w_u$) & - & 1\\
Adjoint boundary condition ($w_{\rm bc}^\lambda$) & - & 1\\
Adjoint terminal condition ($w_{\rm T}^\lambda$) & - & 1\\
\bottomrule
\end{tabular}
\end{table}

The entries in the two method columns multiply different loss functionals and
therefore should not be compared as numerical quantities across columns. In the
one-dimensional paired-endpoint implementation, all boundary losses in both
methods are empirical averages over the $2N_{\rm bc}$ endpoint-time samples, as
specified in Section~\ref{sec:method}.

The tracking and Tikhonov terms in the direct loss are not tunable penalty
terms: they use the fixed coefficients $|Q|/2$ and $\alpha|\Sigma|/2$ from
\eqref{eq:direct_loss}. For the present tests, $|Q|=1$ and $|\Sigma|=2$ under
counting measure on $\partial\Omega=\{0,1\}$.

\subsection{Example 1: cubic reaction–diffusion equation}

\subsubsection{Problem definition}

The first test is a cubic reaction-diffusion equation with $f(y)=y^3$. We prescribe
\begin{equation}
y^*(x,t)=\sin(\pi x)e^{-t},\qquad u^*(0,t)=u^*(1,t)=0,
\end{equation}
and $y_0(x)=\sin(\pi x)$. The manufactured source is
\begin{equation}
\label{eq:source_ex1}
g(x,t)=(-1+\nu\pi^2)\sin(\pi x)e^{-t}+\sin^3(\pi x)e^{-3t}.
\end{equation}
Choosing $y_d=y^*$ makes $(y^*,u^*)$ feasible with $J(y^*,u^*)=0$; hence it is a global optimum. The exact adjoint is $\lambda^*=0$.

\subsubsection{Direct and Indirect results}

Figure~\ref{fig:ex1_state} compares the exact state and the two PINN approximations. The direct and indirect state root-mean-square errors (RMSEs) are $2.34\times10^{-3}$ and $6.80\times10^{-4}$, respectively. Their maximum pointwise state errors are $1.648\times10^{-2}$ and $2.551\times10^{-3}$.

\begin{figure}[H]
\centering
\includegraphics[width=0.98\textwidth]{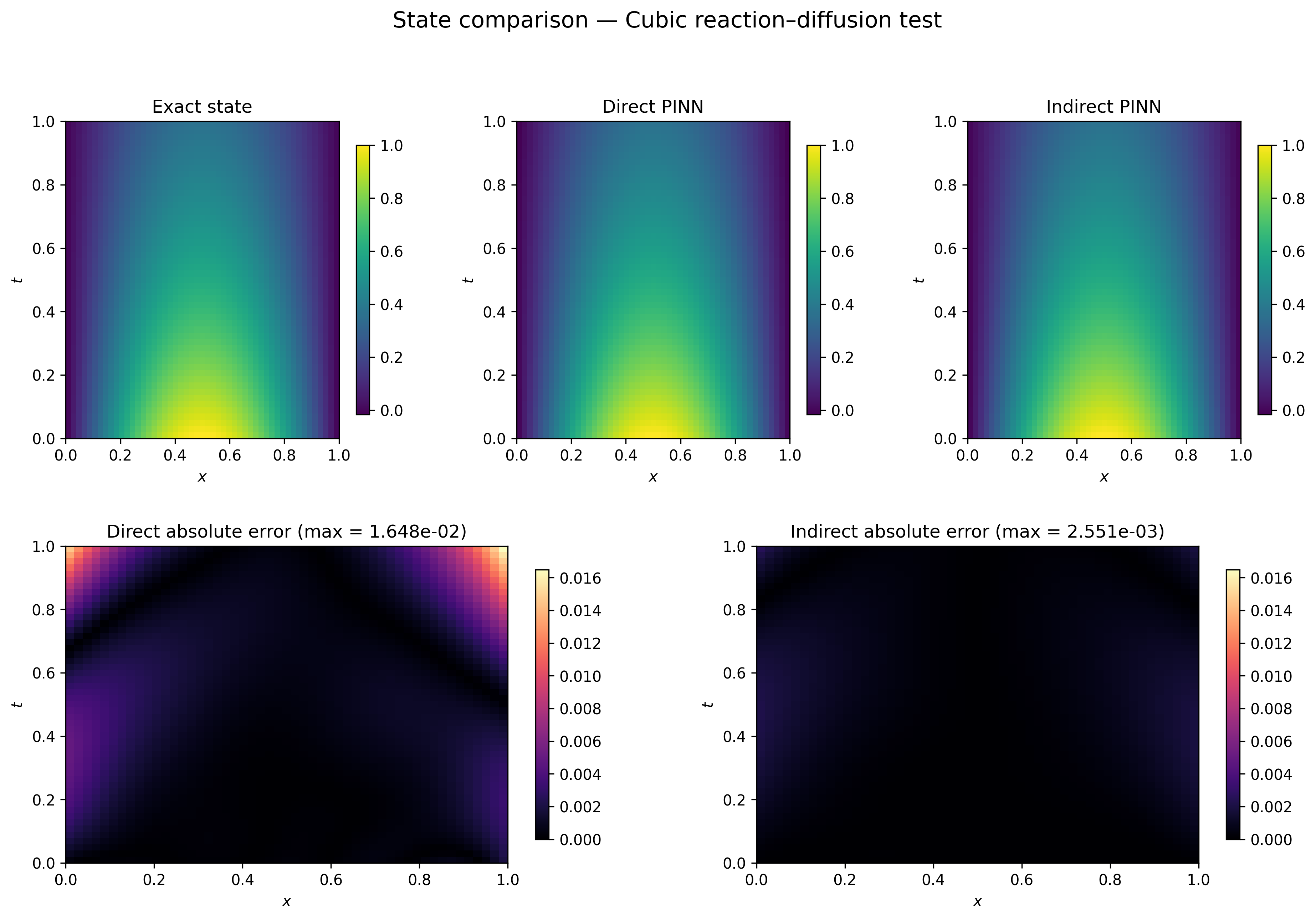}
\caption{State approximation and absolute error for the cubic reaction-diffusion test. Both error panels use the same color scale.}
\label{fig:ex1_state}
\end{figure}

The boundary-control RMSE decreases from $5.98\times10^{-3}$ for the direct method to $1.43\times10^{-3}$ for the indirect method; see Figures~\ref{fig:ex1_control}-\ref{fig:ex1_metrics}.

\begin{figure}[H]
\centering
\includegraphics[width=0.92\textwidth]{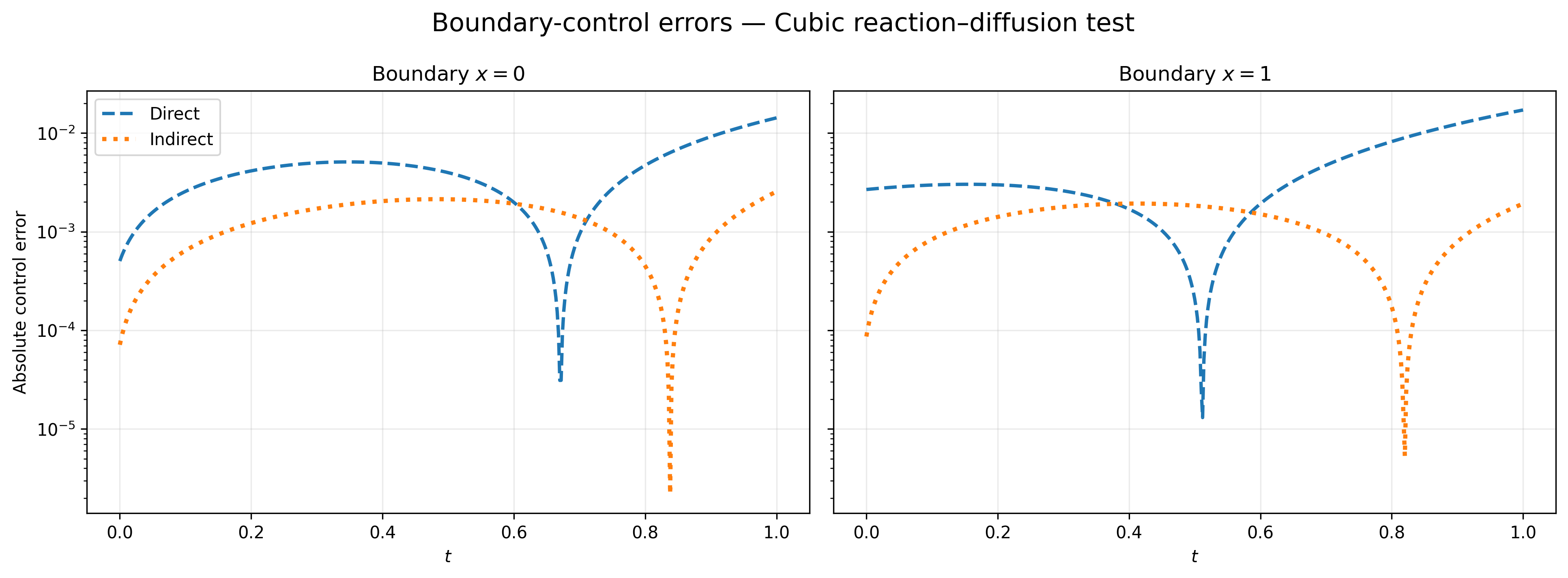}
\caption{Absolute boundary-control errors on the left and right endpoints for Example~1.}
\label{fig:ex1_control}
\end{figure}

\begin{figure}[H]
\centering
\includegraphics[width=0.72\textwidth]{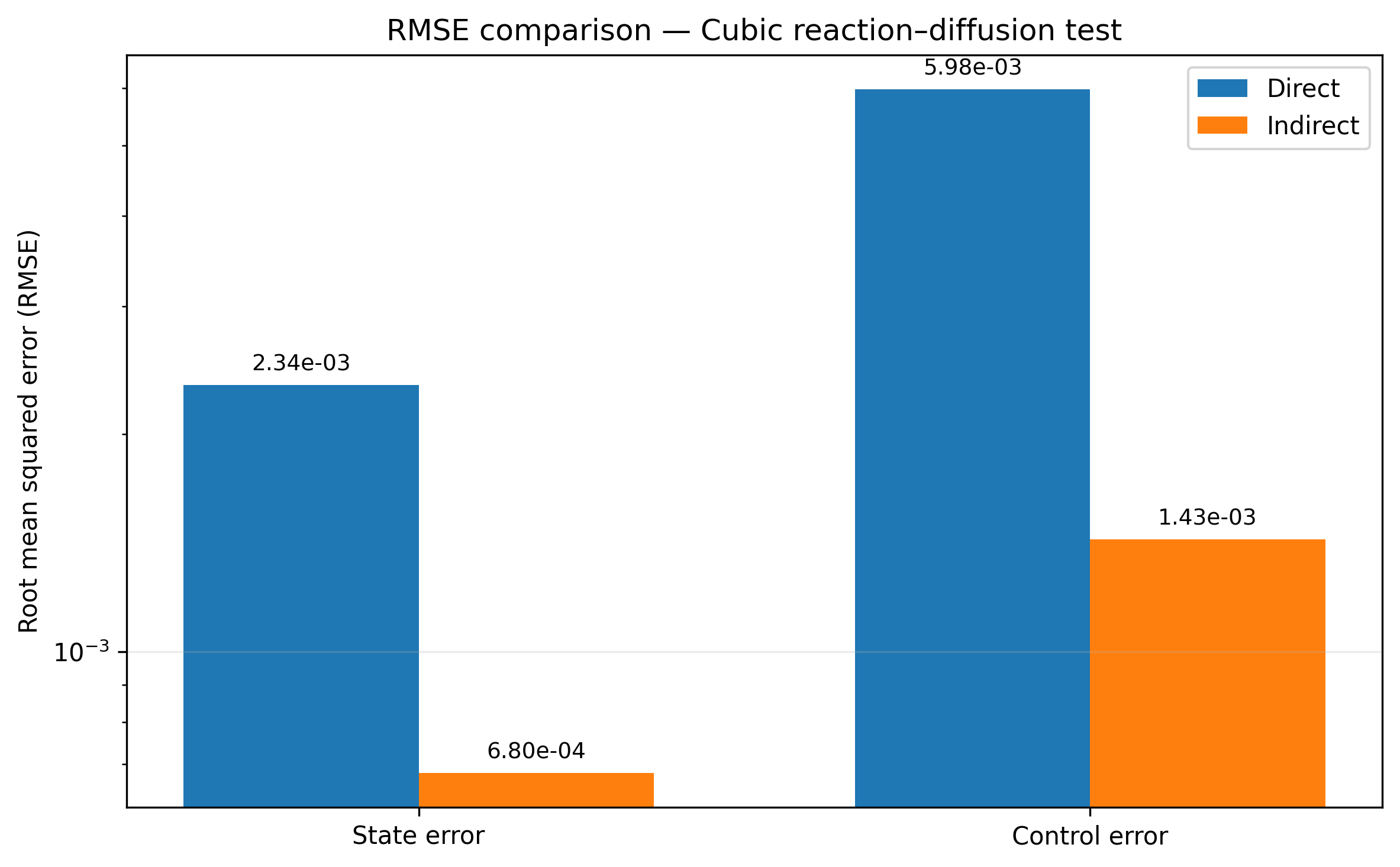}
\caption{RMSE comparison for Example~1.}
\label{fig:ex1_metrics}
\end{figure}

\subsubsection{Discussion}

The indirect formulation reduces the state RMSE by approximately $70.9\%$ and the control RMSE by approximately $76.1\%$. Its largest adjoint error is $1.849\times10^{-4}$, consistent with the exact value $\lambda^*=0$ (Figure~\ref{fig:ex1_adjoint}). The verification maps in Figure~\ref{fig:ex1_residuals} show generally smaller state residuals for the indirect method. The stationarity residual remains mainly of order $10^{-4}$ and decreases near the terminal time (Figure~\ref{fig:ex1_stationarity}).

\begin{figure}[H]
\centering
\includegraphics[width=0.82\textwidth]{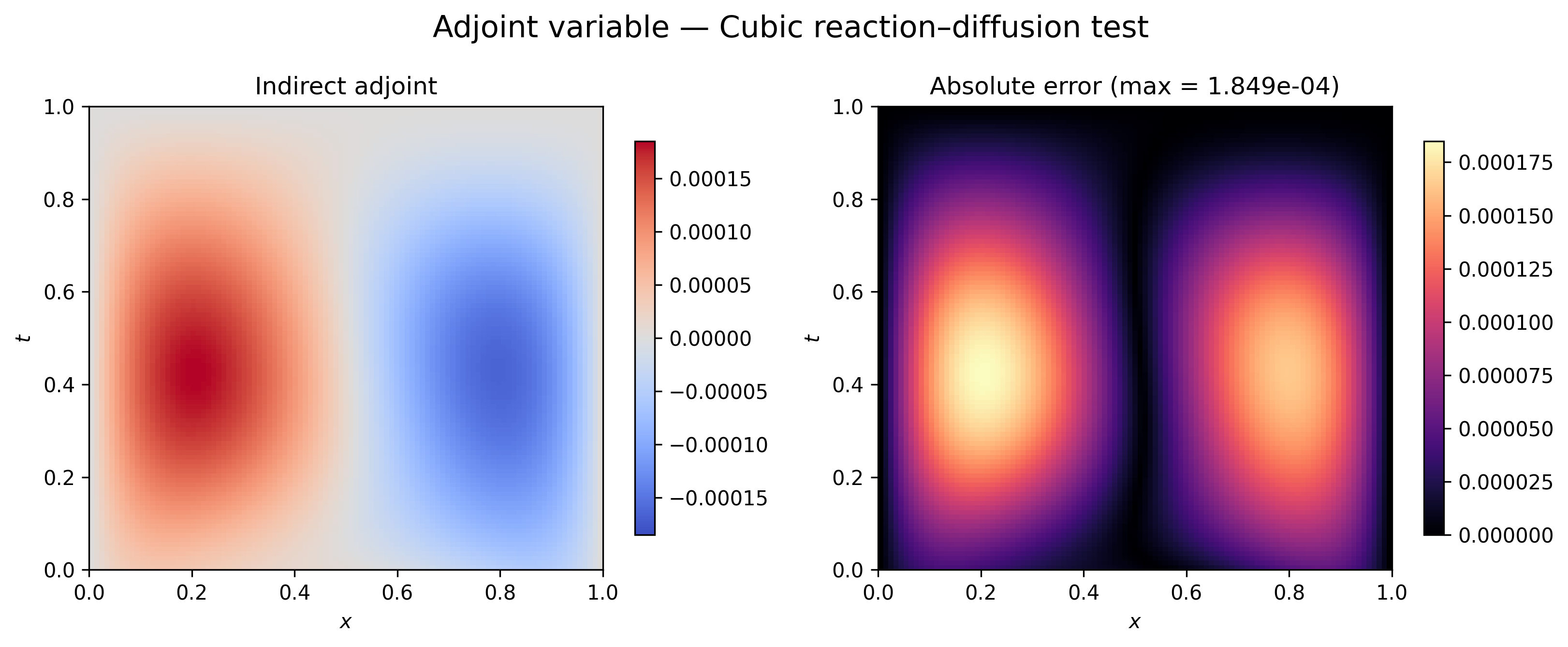}
\caption{Indirect adjoint approximation and absolute error for Example~1.}
\label{fig:ex1_adjoint}
\end{figure}

\begin{figure}[H]
\centering
\includegraphics[width=0.98\textwidth]{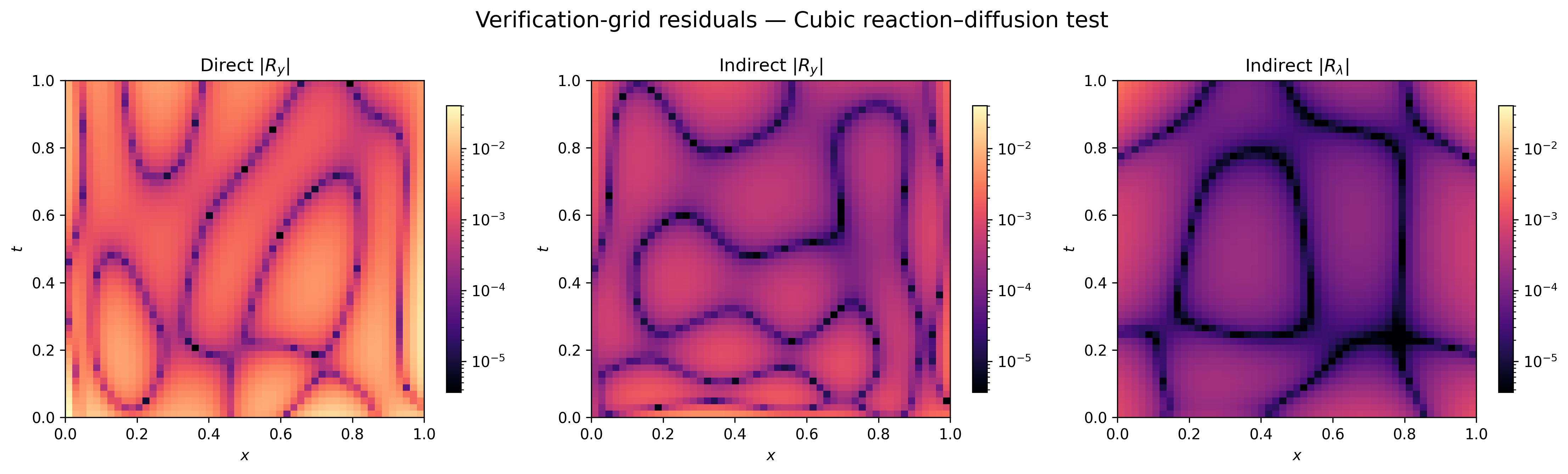}
\caption{State and adjoint residuals evaluated on an independent verification grid for Example~1. A common logarithmic color scale is used.}
\label{fig:ex1_residuals}
\end{figure}

\begin{figure}[H]
\centering
\includegraphics[width=0.74\textwidth]{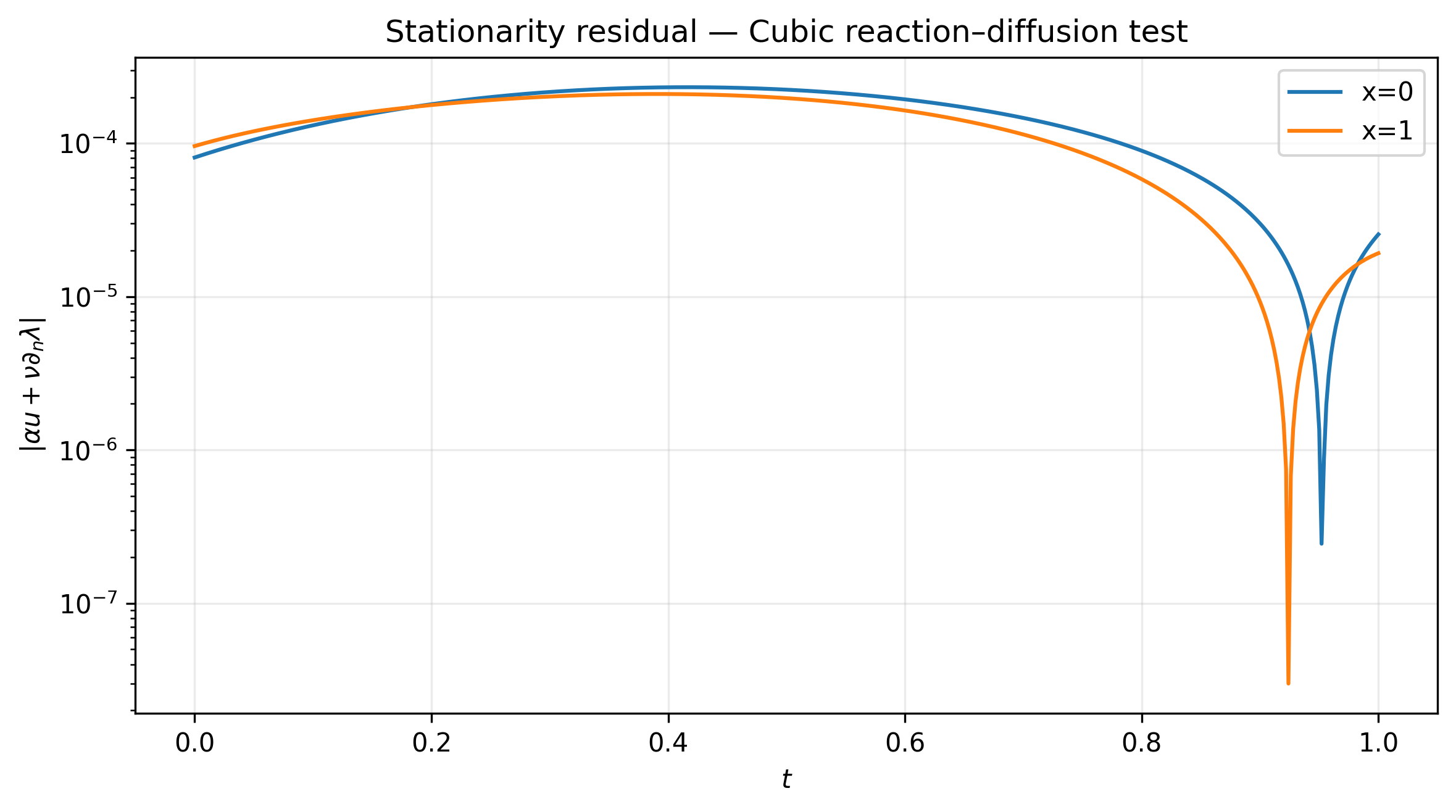}
\caption{Stationarity residual $|\alpha u+\nu\partial_n\lambda|$ at both endpoints for Example~1.}
\label{fig:ex1_stationarity}
\end{figure}

\subsection{Example 2: Linear boundary-control problem}

\subsubsection{Construction of the exact KKT solution}

The second test uses the linear heat equation, $f(y)=0$, and is constructed so that both the control and adjoint are nonzero. Let $A=0.05$ and prescribe
\begin{align}
\lambda^*(x,t)&=A\,x(1-x)\sin(\pi t),\\
u^*(0,t)=u^*(1,t)&=\frac{\nu A}{\alpha}\sin(\pi t)=0.5\sin(\pi t),\\
y^*(x,t)&=0.5\sin(\pi t)+\sin(\pi x)e^{-t}.
\end{align}
The corresponding initial datum is $y_0(x)=y^*(x,0)=\sin(\pi x)$.
The source and desired trajectory are then defined from the state and adjoint equations:
\begin{align}
g(x,t)&=0.5\pi\cos(\pi t)+(-1+\nu\pi^2)\sin(\pi x)e^{-t},\\
y_d(x,t)&=y^*(x,t)-A\pi x(1-x)\cos(\pi t)+2\nu A\sin(\pi t).
\end{align}
The triplet $(y^*,u^*,\lambda^*)$ satisfies the complete optimality system. Since the dynamics are linear and the objective is strictly convex in the control, it is the unique optimum.

\subsubsection{Direct and Indirect results}

Figure~\ref{fig:ex2_state} shows close agreement with the exact state. The direct and indirect state RMSEs are $1.65\times10^{-3}$ and $1.38\times10^{-3}$; the corresponding maximum errors are $8.982\times10^{-3}$ and $5.604\times10^{-3}$.

\begin{figure}[H]
\centering
\includegraphics[width=0.98\textwidth]{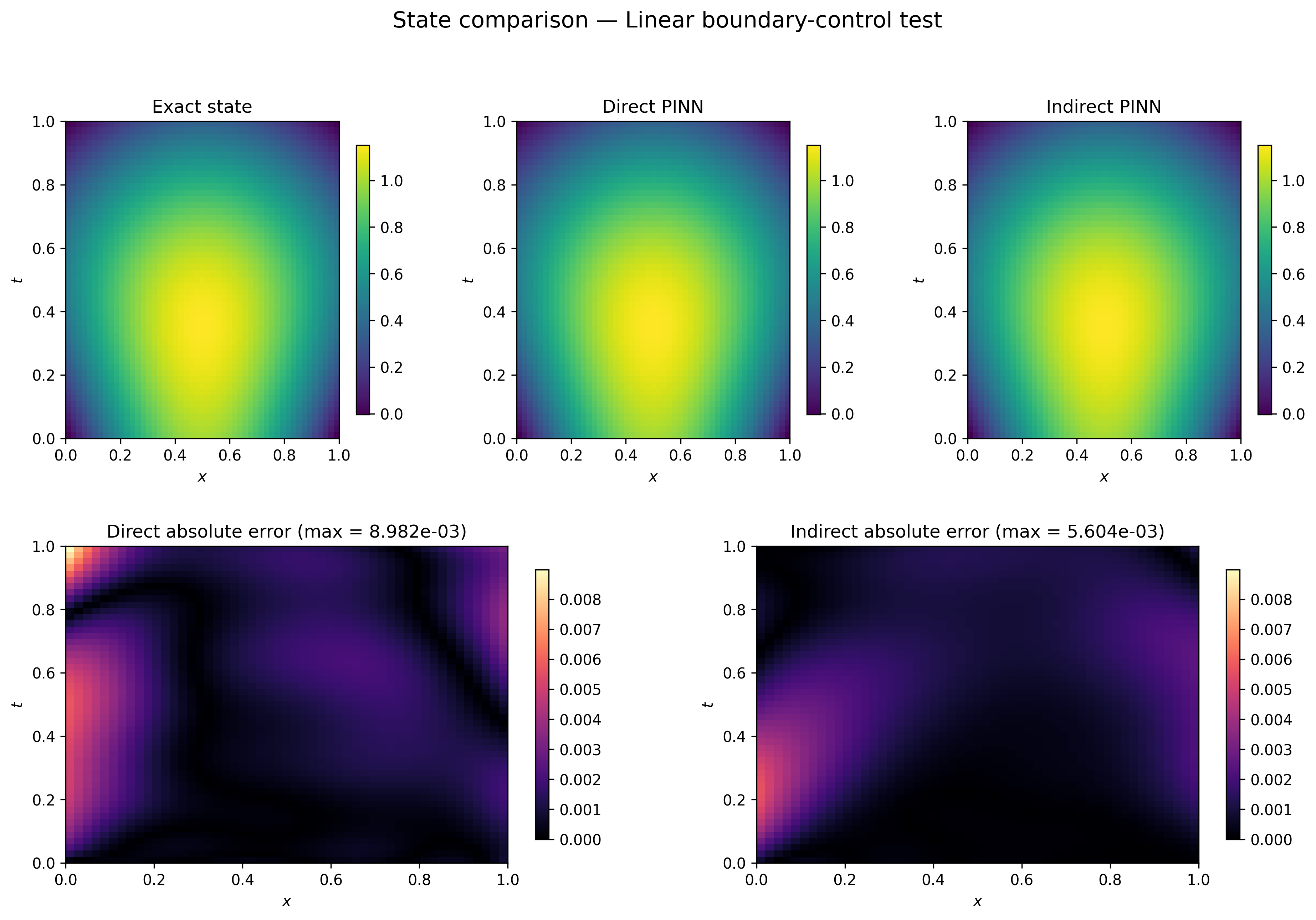}
\caption{State approximation and absolute error for the nontrivial manufactured KKT test.}
\label{fig:ex2_state}
\end{figure}

The control RMSE decreases from $4.08\times10^{-3}$ to $2.55\times10^{-3}$ (Figures~\ref{fig:ex2_control} and~\ref{fig:ex2_metrics}). The indirect adjoint has maximum absolute error $6.079\times10^{-4}$ (Figure~\ref{fig:ex2_adjoint}).

\begin{figure}[H]
\centering
\includegraphics[width=0.92\textwidth]{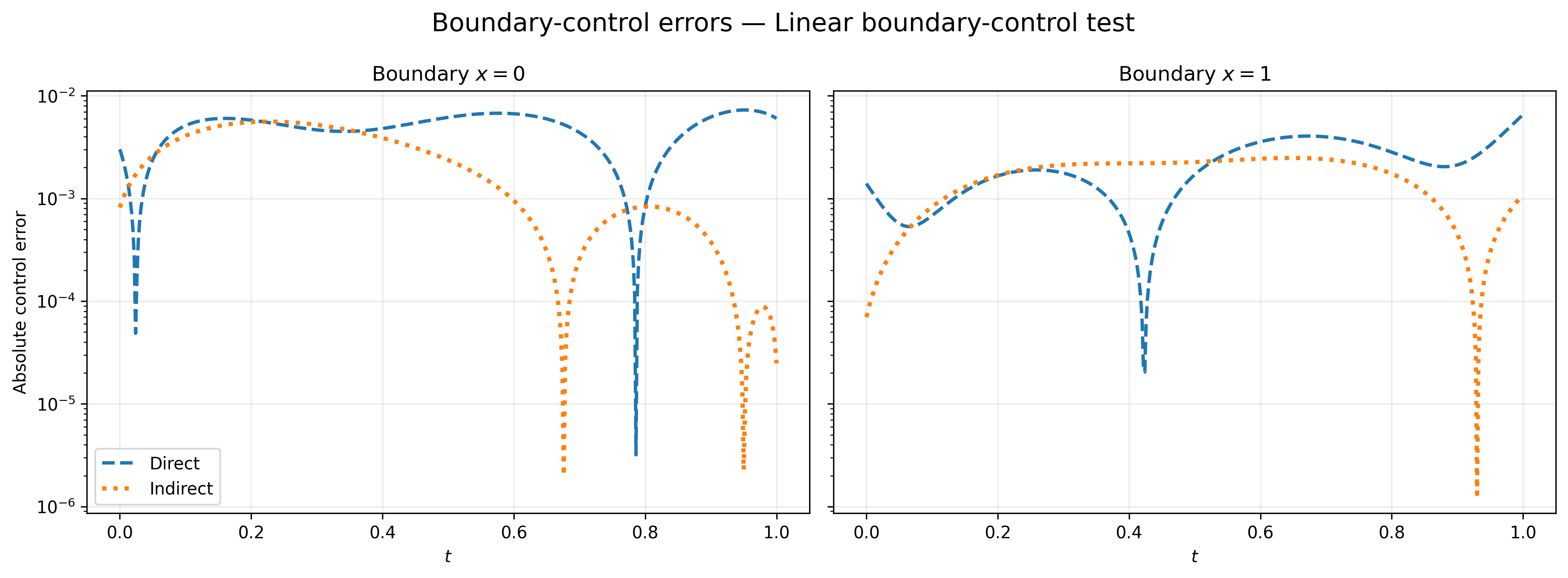}
\caption{Absolute boundary-control errors for Example~2.}
\label{fig:ex2_control}
\end{figure}

\begin{figure}[H]
\centering
\includegraphics[width=0.82\textwidth]{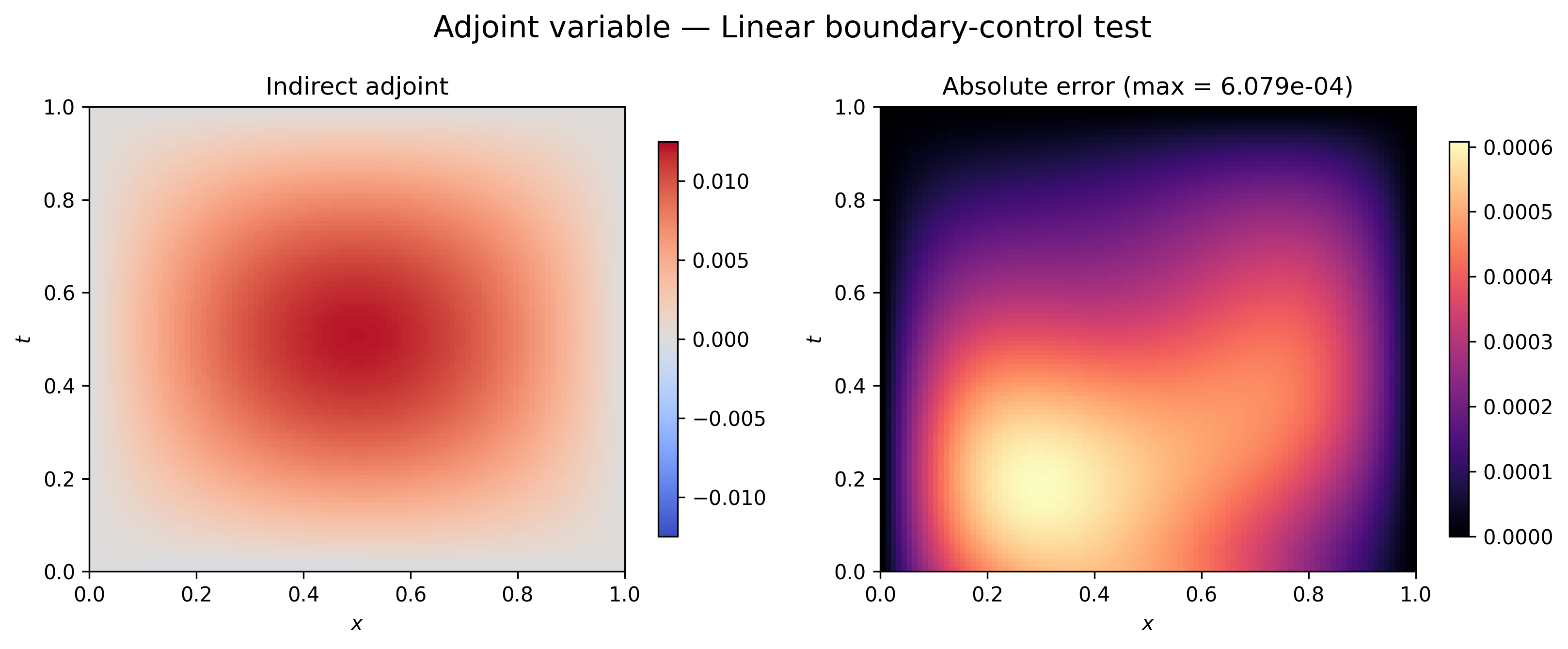}
\caption{Indirect adjoint approximation and absolute error for Example~2.}
\label{fig:ex2_adjoint}
\end{figure}

\begin{figure}[H]
\centering
\includegraphics[width=0.72\textwidth]{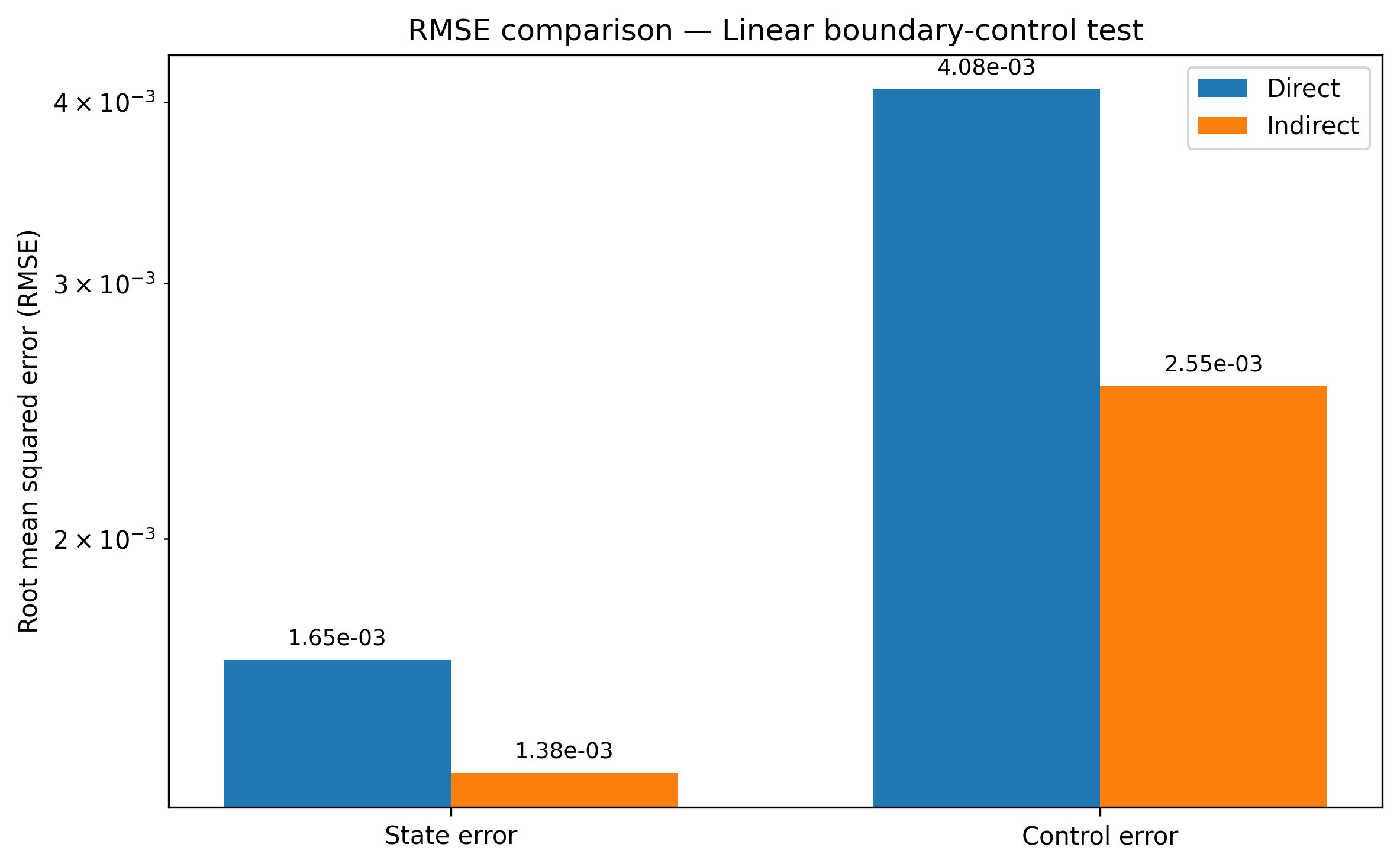}
\caption{RMSE comparison for Example~2.}
\label{fig:ex2_metrics}
\end{figure}

\subsubsection{Discussion}

The indirect method reduces the state RMSE by approximately $16.4\%$ and the control RMSE by approximately $37.5\%$. Unlike Example~1, this test has a nonzero control and a nontrivial adjoint, so the backward adjoint and stationarity equations are tested with nonzero solution fields. The control error is not pointwise smaller at every time, but its global RMSE is lower. Figures~\ref{fig:ex2_residuals} and \ref{fig:ex2_stationarity} show that the learned fields remain consistent with the PDE and first-order optimality conditions on independent verification points.

\begin{figure}[H]
\centering
\includegraphics[width=0.98\textwidth]{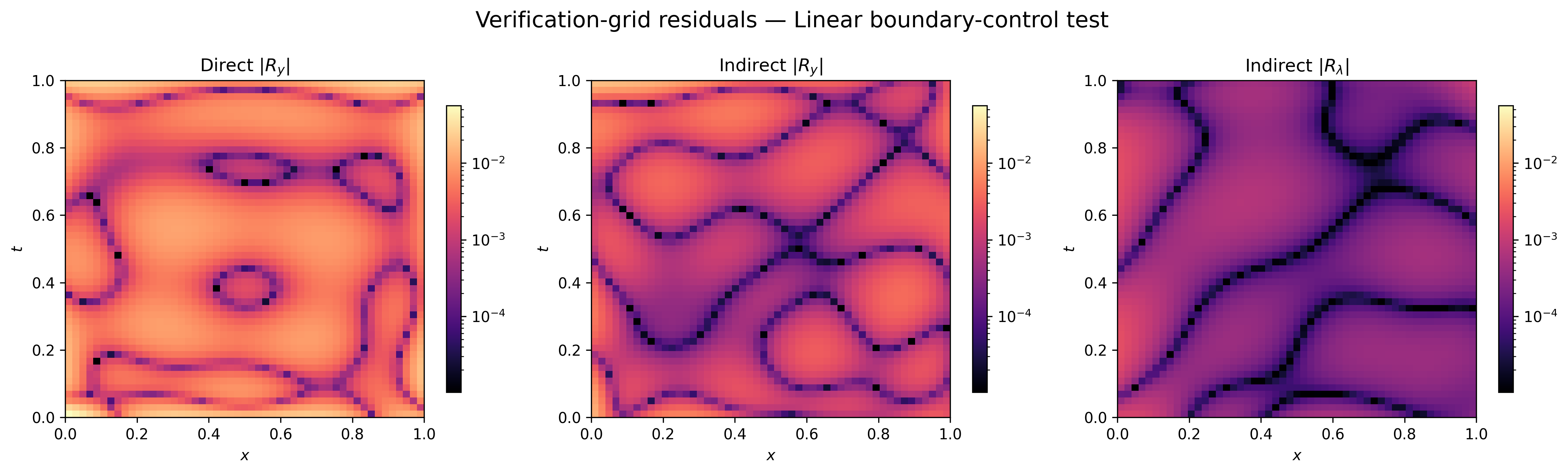}
\caption{Verification-grid state and adjoint residuals for Example~2.}
\label{fig:ex2_residuals}
\end{figure}

\begin{figure}[H]
\centering
\includegraphics[width=0.74\textwidth]{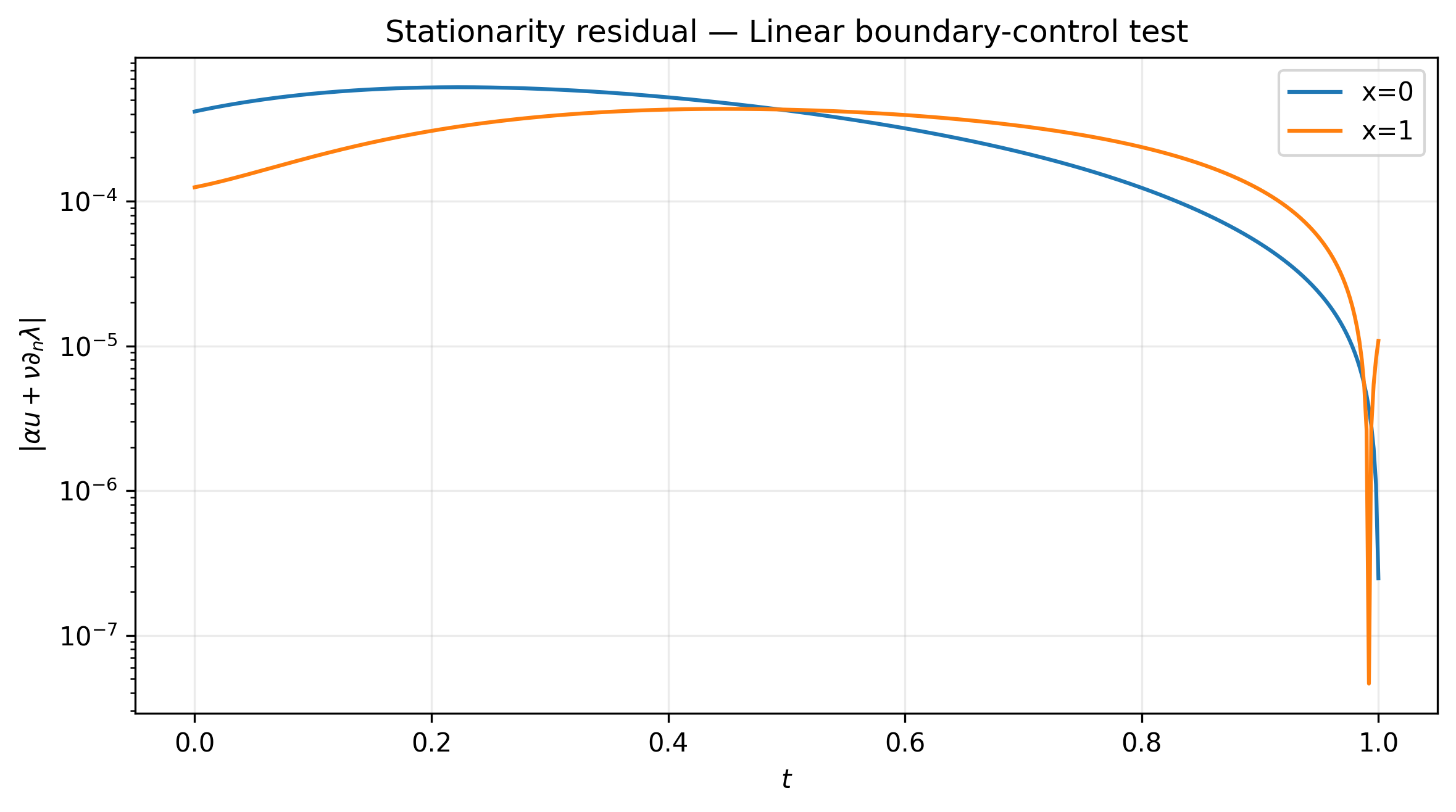}
\caption{Stationarity residual at both endpoints for Example~2.}
\label{fig:ex2_stationarity}
\end{figure}

\subsection{Overall comparison}

\begin{table}[H]
\centering
\caption{Summary of the numerical errors.}
\label{tab:overall_results}
\begin{tabular}{llccc}
\toprule
Example & Method & State RMSE & Control RMSE & Maximum state error\\
\midrule
1 & Direct & $2.34\times10^{-3}$ & $5.98\times10^{-3}$ & $1.648\times10^{-2}$\\
1 & Indirect & $6.80\times10^{-4}$ & $1.43\times10^{-3}$ & $2.551\times10^{-3}$\\
2 & Direct & $1.65\times10^{-3}$ & $4.08\times10^{-3}$ & $8.982\times10^{-3}$\\
2 & Indirect & $1.38\times10^{-3}$ & $2.55\times10^{-3}$ & $5.604\times10^{-3}$\\
\bottomrule
\end{tabular}
\end{table}

For the present runs, the indirect formulation gives lower state and control RMSE in both examples. The improvement is strongest in Example~1, while Example~2 confirms that the method also recovers a nontrivial boundary control and adjoint. These single-seed results provide an accuracy comparison rather than an equal-computational-budget or statistical comparison.
\section{Conclusions}
\label{sec:conclusions}

We extended the direct-indirect PINN framework to a semilinear parabolic optimal-control problem with Dirichlet boundary actuation. A single network $U_\psi(x,t)$ represents the control through its restriction to the entire lateral boundary. In the indirect method, the state, adjoint, and control are represented by unconstrained networks, with the boundary coupling $y=u$, the homogeneous adjoint boundary condition, and the adjoint terminal condition all imposed as soft penalty terms in the training loss rather than enforced architecturally.

Two manufactured tests were considered. 
The cubic reaction-diffusion test reproduces the original zero-control setting, whereas the linear KKT test contains a nonzero control and adjoint.
Independent verification-grid residuals also showed consistency with the state, adjoint, and stationarity equations.

The error estimates remain conditional on well-posedness, regularity, stability, residual-compatible network approximability, the uniform H\"older control of the boundary/terminal residuals in (H5), and the local-neighborhood assumption (H6) used to absorb the nonlinear Taylor remainders. In particular, no efficiency result is claimed for the meshfree residual indicator, and the sufficient condition $\alpha>\alpha_0$ is not numerically certified in the experiments. Future work should include repeated random seeds, convergence studies with increasing collocation sets, genuinely low-regularity Dirichlet data, nonlinear tests with nontrivial controls, and empirical verification of (H5)-(H6) (e.g.\ via weight control and post-hoc checks on a held-out grid).

\appendix
\section{Appendix: Auxiliary results}
\label{app:cited-results}

\subsection{Parabolic Schauder estimate with general boundary and terminal data}
\label{app:schauder}

The following lemma is the single Schauder-type estimate used throughout
Section~\ref{sec:error-analysis}: with $(g_1,v_0)=(u^*,y_0)$ it gives state
regularity, and with homogeneous data it gives adjoint regularity in
Proposition~\ref{prop:regularity};
with inhomogeneous data $(g_1,v_0)=(-R^\lambda_{\rm bc},-R^\lambda_T)$ it gives
Step~2 of Theorem~\ref{thm:quant-stability}. Stating it once, for general data,
removes the need for two separate citation routes.

\begin{lemma}[Parabolic Schauder estimate]
\label{lem:schauder}
Let $\Omega\subset\mathbb R^d$ have $C^{2+\beta}$ boundary for some
$\beta\in(0,1)$, let $T>0$, $Q=\Omega\times(0,T)$, $\Sigma=\partial\Omega\times(0,T)$.
Let $a\in C^{\beta,\beta/2}(\overline Q)$, and suppose
\[
h\in C^{\beta,\beta/2}(\overline Q),\qquad
g_1\in C^{2+\beta,1+\beta/2}(\Sigma),\qquad
v_0\in C^{2+\beta}(\overline\Omega),
\]
satisfy the parabolic compatibility conditions at $t=0$,
\[
v_0=g_1(\cdot,0),\qquad
\partial_tg_1(\cdot,0)=\nu\Delta v_0-a(\cdot,0)v_0+h(\cdot,0)
\quad\text{on }\partial\Omega.
\]
Then the linear problem
\[
\partial_t v-\nu\Delta v+av=h\ \text{ in }Q,\qquad v=g_1\ \text{ on }\Sigma,\qquad v(\cdot,0)=v_0\ \text{ in }\Omega,
\]
has a unique solution $v\in C^{2+\beta,1+\beta/2}(\overline Q)$, and
\begin{equation}
\label{eq:schauder-general}
\|v\|_{C^{2+\beta,1+\beta/2}(\overline Q)}
\;\le\;
C_{\rm Sch}\Bigl(\|h\|_{C^{\beta,\beta/2}(\overline Q)}
+\|g_1\|_{C^{2+\beta,1+\beta/2}(\Sigma)}
+\|v_0\|_{C^{2+\beta}(\overline\Omega)}\Bigr),
\end{equation}
with $C_{\rm Sch}=C_{\rm Sch}(\nu,\Omega,T,\beta,\|a\|_{C^{\beta,\beta/2}(\overline Q)})$.
\end{lemma}

\begin{proof}
Existence and uniqueness of a classical solution under exactly these
hypotheses is the parabolic Schauder theorem; see
\cite[Ch.~3, Thm.~5]{friedman1964parabolic}, reproduced as
\cite[Thm.~F.1]{shin2020convergence}. The quantitative bound
\eqref{eq:schauder-general} is the accompanying Schauder \emph{estimate}
(as opposed to bare existence), obtained from the same theory by the
standard localization/freezing-of-coefficients argument: cover
$\overline Q$ by finitely many parabolic cylinders on which $a$ is
approximately constant, apply the interior and boundary Schauder estimates
for the constant-coefficient operator $\partial_t-\nu\Delta+a(x_0,t_0)$ on
each cylinder, and reassemble via a partition of unity, absorbing the
lower-order commutator terms (which are $O(\beta)$-small on each cylinder
by continuity of $a$) into the left-hand side. This is the standard elliptic-type
bootstrap for parabolic Schauder theory and does not depend on
whether $g_1,v_0$ vanish.
\end{proof}

Proposition~\ref{prop:regularity} is Lemma~\ref{lem:schauder} applied with
$a=f'(y^*)$, $h=g-f(y^*)+f'(y^*)y^*$, $g_1=u^*$, $v_0=y_0$ (and, after time
reversal $\tau=T-t$, to $\lambda^*$ with homogeneous data), giving
$y^*,\lambda^*\in C^{2+\beta,1+\beta/2}(\overline Q)$.

Step~2 of Theorem~\ref{thm:quant-stability} applies Lemma~\ref{lem:schauder}
only to the component $e_\lambda^{(bd)}$ with zero interior forcing and
boundary/terminal data $(-R^\lambda_{\rm bc},-R^\lambda_T)$. It gives
\[
\|e_\lambda^{(bd)}\|_{C^{2+\beta,1+\beta/2}(\overline Q)}
\;\le\;
C_{\rm Sch}\Bigl(\|R^\lambda_{\rm bc}\|_{C^{2+\beta,1+\beta/2}(\Sigma)}
+\|R^\lambda_T\|_{C^{2+\beta}(\overline\Omega)}\Bigr),
\]
and hypothesis (H5) bounds the two data terms by $M_{\rm bc}$. The
homogeneous-data component $e_\lambda^{(F)}$, whose forcing is
$F_\lambda=-(\kappa+1)e_y-\mathcal R_\lambda\in L^2(Q)$, is handled
separately by Lemma~\ref{lem:hidden-reg}; no Hölder regularity of $e_y$ is
used.

\subsection{Extraction of the normal derivative}
\label{app:trace}

Because Lemma~\ref{lem:schauder} places $e_\lambda^{(bd)}$ directly in the classical
space $C^{2+\beta,1+\beta/2}(\overline Q)$ rather than only in an $H^{2,1}(Q)$
Sobolev space, the normal derivative $\partial_n e_\lambda^{(bd)}$ on $\Sigma$ is
defined classically, and no separate trace theorem is required.

\begin{lemma}[Normal-derivative bound]
\label{lem:normal-derivative}
If $v\in C^{2+\beta,1+\beta/2}(\overline Q)$, then $\partial_n v\big|_\Sigma$
exists in the classical sense, $\partial_n v\in C^{1+\beta,(1+\beta)/2}(\Sigma)\subset C^0(\Sigma)$, and
\[
\|\partial_n v\|_{L^2(\Sigma)}\;\le\;|\Sigma|^{1/2}\,\|\partial_n v\|_{C^0(\Sigma)}\;\le\;C\,\|v\|_{C^{2+\beta,1+\beta/2}(\overline Q)},
\]
with $C=C(\Omega,T)$ depending on the $C^{2+\beta}$ regularity of $\partial\Omega$ and on the length of the time interval.
\end{lemma}

\begin{proof}
By definition of $C^{2+\beta,1+\beta/2}(\overline Q)$, all first and second
spatial derivatives of $v$ exist and are continuous, indeed Hölder
continuous of order $1+\beta$ jointly in $(x,t)$ (in $x$) up to $\overline Q$;
in particular $\nabla v\in C^{1+\beta,(1+\beta)/2}(\overline Q)$. Since
$\partial\Omega\in C^{2+\beta}$, the outward unit normal
$n:\partial\Omega\to\mathbb R^d$ is $C^{1+\beta}$, so
$\partial_n v(z,t):=n(z)\cdot\nabla v(z,t)$ is a well-defined pointwise
(classical, not distributional) evaluation for $(z,t)\in\Sigma$, and is
Hölder continuous of order $\min(1+\beta,1+\beta)=1+\beta$ in $z$, order
$(1+\beta)/2$ in $t$, by the product rule for Hölder functions applied to
$n(z)\cdot\nabla v(z,t)$. The sup-norm bound
$\|\partial_n v\|_{C^0(\Sigma)}\le\|\nabla v\|_{C^0(\overline Q)}\le\|v\|_{C^{2+\beta,1+\beta/2}(\overline Q)}$
is immediate from the definition of the Hölder norm, and $|\Sigma|<\infty$
since $\partial\Omega$ is compact and $T<\infty$.
\end{proof}

Combined with Lemma~\ref{lem:schauder}, Lemma~\ref{lem:normal-derivative}
controls the normal derivative of the boundary/terminal-data component
$e_\lambda^{(bd)}$ used in Step~2 of Theorem~\ref{thm:quant-stability}:
\[
\|\partial_n e_\lambda^{(bd)}\|_{L^2(\Sigma)}
\;\le\;2C\,C_{\rm Sch}M_{\rm bc}.
\]
This is the source of the additive $M_{\rm bc}$ term in
\eqref{eq:stability-bound}; Remark~\ref{rem:mbc-additive} explains both this
dependence and the Gagliardo-Nirenberg refinement that would be needed to
make it sensitive to the $L^2(\Sigma)$ and $L^2(\Omega)$ sizes of the
boundary and terminal residuals, respectively.

\subsection{Hidden regularity via the Rellich-Pohozaev multiplier identity}
\label{app:hidden-reg}

Step~2 of the proof of Theorem~\ref{thm:quant-stability} needs the
following fact: given $L^2(Q)$ forcing and \emph{homogeneous} Dirichlet
data, the resulting solution has a normal derivative that is still in
$L^2(\Sigma)$ - a genuine regularity gain over what the elementary energy
estimate alone would give. This phenomenon is classically called
\emph{hidden regularity} in the Dirichlet boundary control literature
(\cite[Ch.~4]{lionsmagenes1972nonhomogeneous};
\cite{lasieckatriggiani2000control}), where it is typically obtained from
abstract maximal-regularity and trace theorems. We instead give a
self-contained, elementary proof via the classical Rellich-Pohozaev
multiplier identity, so that the estimate does not rest on citing those
results as black boxes.

\begin{lemma}[Hidden regularity]
\label{lem:hidden-reg}
Let $\Omega\subset\mathbb R^d$ be bounded with $C^2$ boundary, $a\in
L^\infty(Q)$, $F\in L^2(Q)$, and let $w$ solve
\[
\partial_tw-\nu\Delta w+aw=F\ \text{in }Q,\qquad w=0\ \text{on }\Sigma,\qquad w(\cdot,0)=0\ \text{in }\Omega.
\]
Then $w\in C([0,T];H_0^1(\Omega))\cap L^2(0,T;H^2(\Omega))\cap H^1(0,T;L^2(\Omega))$, and:
\begin{enumerate}
\item[(R1)] $\displaystyle \|w\|_{L^2(0,T;H^1(\Omega))}^2\;\le\;TC_1\,\|F\|_{L^2(Q)}^2$;
\item[(R2)] $\displaystyle \|\partial_nw\|_{L^2(\Sigma)}^2\;\le\;C_2\,\|F\|_{L^2(Q)}^2$;
\end{enumerate}
with $C_1,C_2>0$ depending only on $\nu,\Omega,T,\|a\|_{L^\infty(Q)}$.
\end{lemma}

\begin{proof}
\emph{(R1), $L^2(Q)$-based maximal regularity.} Let $\{e_k\}\subset H_0^1(\Omega)\cap H^2(\Omega)$ be Dirichlet-Laplacian eigenfunctions, $-\Delta e_k=\mu_k e_k$, orthonormal in $L^2(\Omega)$. Let $w_m=\sum_{k\le m}c_k^m(t)e_k$ solve the Galerkin system $(\partial_tw_m,e_k)+\nu(\nabla w_m,\nabla e_k)+(aw_m,e_k)=(F,e_k)$, $k=1,\dots,m$, with $w_m(0)=0$.

Testing with $\partial_tw_m$ and using $|(aw_m,\partial_tw_m)|\le\tfrac12\|\partial_tw_m\|_{L^2(\Omega)}^2+\tfrac12\|a\|_{L^\infty(Q)}^2\|w_m\|_{L^2(\Omega)}^2$,
\[
\tfrac12\|\partial_tw_m\|_{L^2(\Omega)}^2+\tfrac\nu2\frac{d}{dt}\|\nabla w_m\|_{L^2(\Omega)}^2\;\le\;\|F\|_{L^2(\Omega)}^2+\|a\|_{L^\infty(Q)}^2\|w_m\|_{L^2(\Omega)}^2.
\]
By Poincar\'e's inequality and Gr\"onwall's inequality applied to $t\mapsto\|\nabla w_m(t)\|_{L^2(\Omega)}^2$ (using $w_m(0)=0$), integrating over $(0,T)$ gives
\begin{equation}
\label{eq:energy1}
\sup_{t\in[0,T]}\|\nabla w_m(t)\|_{L^2(\Omega)}^2+\int_0^T\|\partial_tw_m\|_{L^2(\Omega)}^2\,dt\;\le\;C(\Omega,\nu,\|a\|_{L^\infty(Q)},T)\,\|F\|_{L^2(Q)}^2.
\end{equation}
Testing instead with $-\Delta w_m$ (well defined in $\mathrm{span}\{e_1,\dots,e_m\}$, with $(\partial_tw_m,-\Delta w_m)=\tfrac12\frac{d}{dt}\|\nabla w_m\|_{L^2(\Omega)}^2$),
\[
\begin{aligned}
\tfrac12\frac{d}{dt}\|\nabla w_m\|_{L^2(\Omega)}^2
+\nu\|\Delta w_m\|_{L^2(\Omega)}^2
&=(F,-\Delta w_m)-(aw_m,-\Delta w_m)\\
&\le\tfrac\nu2\|\Delta w_m\|_{L^2(\Omega)}^2
+\tfrac1\nu\bigl(\|F\|_{L^2(\Omega)}
+\|a\|_{L^\infty(Q)}\|w_m\|_{L^2(\Omega)}\bigr)^2,
\end{aligned}
\]
so integrating over $(0,T)$ and using \eqref{eq:energy1}, $\int_0^T\|\Delta w_m\|_{L^2(\Omega)}^2\,dt\le C(\|F\|_{L^2(Q)}^2)$. Since $\partial\Omega\in C^2$, elliptic $H^2$-regularity for the Dirichlet Laplacian gives $\|w_m\|_{H^2(\Omega)}\le C(\Omega)\|\Delta w_m\|_{L^2(\Omega)}$. Passing $m\to\infty$ (weak limits along a subsequence, weak lower semicontinuity of norms, standard identification of the limit as the unique weak solution) gives (R1) with $C_1$ as claimed.

\emph{(R2), multiplier identity.} Since $\partial\Omega\in C^2$, fix $m\in C^1(\overline\Omega;\mathbb R^d)$ with $m=n$ (outward unit normal) on $\partial\Omega$ (e.g.\ $m=\nabla\zeta$ for a $C^2$ extension $\zeta$ of the signed distance to $\partial\Omega$, cut off away from $\partial\Omega$). By (R1), $w(\cdot,t)\in H^2(\Omega)\cap H^1_0(\Omega)$ for a.e.\ $t$, so $\nabla w=(\partial_n w)\,n$ on $\partial\Omega$. Multiply the PDE by $m\cdot\nabla w$ and integrate over $Q$; using $\nu\Delta w=\partial_tw+aw-F$,
\begin{equation}
\label{eq:mult-eq}
-\nu\int_Q\Delta w\,(m\cdot\nabla w)\;=\;-\int_Q(\partial_tw+aw-F)(m\cdot\nabla w).
\end{equation}
On the other hand, integrating by parts in $x$ (for a.e.\ fixed $t$, then in $t$), with $\nabla(m\cdot\nabla w)=(Dm)^T\nabla w+(D^2w)m$ and $m\cdot\nabla(\tfrac12|\nabla w|^2)=\nabla w^T(D^2w)m$,
\[
-\int_\Omega\Delta w\,(m\cdot\nabla w)\,dx
=\int_\Omega\nabla w^T(Dm)\nabla w\,dx+\int_\Omega m\cdot\nabla\bigl(\tfrac12|\nabla w|^2\bigr)\,dx-\int_{\partial\Omega}\partial_nw\,(m\cdot\nabla w)\,d\sigma.
\]
Since $m\cdot n=1$ and $\nabla w=(\partial_nw)n$ on $\partial\Omega$, $m\cdot\nabla w=\partial_nw$ there, so the last term equals $-\int_{\partial\Omega}|\partial_nw|^2\,d\sigma$; integrating the middle term by parts again, $\int_\Omega m\cdot\nabla(\tfrac12|\nabla w|^2)=-\int_\Omega\tfrac12|\nabla w|^2\,{\rm div}\,m+\int_{\partial\Omega}\tfrac12|\partial_nw|^2\,d\sigma$. Combining,
\[
-\int_\Omega\Delta w(m\cdot\nabla w)\,dx=\int_\Omega\Bigl[\nabla w^T(Dm)\nabla w-\tfrac12|\nabla w|^2\,{\rm div}\,m\Bigr]dx-\tfrac12\int_{\partial\Omega}|\partial_nw|^2\,d\sigma.
\]
Multiplying by $\nu$, integrating over $(0,T)$, and equating with \eqref{eq:mult-eq} to solve for the boundary term,
\begin{equation}
\label{eq:rellich}
\tfrac\nu2\int_\Sigma|\partial_nw|^2\;=\;\nu\int_Q\Bigl[\nabla w^T(Dm)\nabla w-\tfrac12|\nabla w|^2\,{\rm div}\,m\Bigr]\;+\;\int_Q(\partial_tw+aw-F)(m\cdot\nabla w).
\end{equation}
By Cauchy-Schwarz, every term on the right of \eqref{eq:rellich} is
bounded by a constant depending only on the data of
Lemma~\ref{lem:hidden-reg} and on the fixed vector field $m$, multiplied by
\[
\|\nabla w\|_{L^2(Q)}^2+\|\partial_tw\|_{L^2(Q)}^2
+\|w\|_{L^2(Q)}^2+\|F\|_{L^2(Q)}^2.
\]
Each term in this sum is controlled by \eqref{eq:energy1}. This gives (R2).
\end{proof}

Eq.~\eqref{eq:step2-final} in the proof of Theorem~\ref{thm:quant-stability}
applies Lemma~\ref{lem:hidden-reg} (after the time reversal $\tau=T-t$)
with the exact forcing $F=F_\lambda=-(\kappa+1)e_y-\mathcal R_\lambda\in
L^2(Q)$. The estimate
$\|F_\lambda\|_{L^2(Q)}\le(1+M_\kappa)\|e_y\|_{L^2(Q)}
+\|\mathcal R_\lambda\|_{L^2(Q)}$ is then used with (R2) for the
normal-derivative bound and with (R1) for the $L^2(0,T;H^1(\Omega))$ bound
on $e_\lambda^{(F)}$ quoted at the end of Step~2.

\subsection{Transposition solution for $L^2(\Sigma)$ Dirichlet data}
\label{app:transposition}

The Schauder estimate of Lemma~\ref{lem:schauder} requires H\"older
boundary data $g_1\in C^{2+\beta,1+\beta/2}(\Sigma)$. In Step~1 of the
proof of Theorem~\ref{thm:quant-stability}, the piece $e_y^{(g)}$ of the
state error is driven by boundary data $e_u-R^y_{\rm bc}$ that is
controlled only in $L^2(\Sigma)$ (Remark~\ref{rem:E-norm-fix}), so a
different, weaker solution concept is needed: the classical
\emph{transposition} (very weak) solution, used for Dirichlet boundary
control precisely to accommodate $L^2(\Sigma)$ data (see
\cite{gong2016dirichlet}). Unlike the usual construction of this solution
concept, which invokes an abstract maximal-regularity and trace theorem
for the dual problem (as in \cite[Ch.~4, \S9-10]{lionsmagenes1972nonhomogeneous}),
we obtain it and its estimate directly from Lemma~\ref{lem:hidden-reg} by
duality, with no further citation needed.

\begin{corollary}[Transposition solution and its $L^2(Q)$ estimate]
\label{cor:transposition}
Let $\Omega\subset\mathbb R^d$ have $C^2$ boundary, $a\in L^\infty(Q)$,
and $g\in L^2(\Sigma)$. There is a unique $v\in L^2(Q)$, called the
transposition solution of
\[
\partial_tv-\nu\Delta v+av=0\ \text{in }Q,\qquad v=g\ \text{on }\Sigma,\qquad v(\cdot,0)=0\ \text{in }\Omega,
\]
defined by
\begin{equation}
\label{eq:transposition-def}
\int_Qv\,\phi\,dx\,dt\;=\;-\nu\int_\Sigma g\,\partial_nz_\phi\,d\sigma\,dt
\qquad\text{for every }\phi\in L^2(Q),
\end{equation}
where $z_\phi$ solves the backward dual problem
\[
-\partial_tz_\phi-\nu\Delta z_\phi+az_\phi=\phi\ \text{in }Q,\qquad z_\phi=0\ \text{on }\Sigma,\qquad z_\phi(\cdot,T)=0\ \text{in }\Omega,
\]
and $v$ satisfies
\begin{equation}
\label{eq:transposition-bound}
\|v\|_{L^2(Q)}\;\le\;C_3\,\|g\|_{L^2(\Sigma)},
\end{equation}
with $C_3=C_3(\nu,\Omega,T,\|a\|_{L^\infty(Q)})$.
\end{corollary}

\begin{proof}
After the time reversal $\tau=T-t$, the backward, zero-terminal-data
problem for $z_\phi$ becomes a forward, zero-initial-data problem of
exactly the type treated in Lemma~\ref{lem:hidden-reg}, with forcing
$\phi\in L^2(Q)$. By (R2) of that lemma,
\begin{equation}
\label{eq:trace-z}
\|\partial_nz_\phi\|_{L^2(\Sigma)}\;\le\;\sqrt{C_2}\,\|\phi\|_{L^2(Q)}.
\end{equation}
Combining \eqref{eq:trace-z} with \eqref{eq:transposition-def}, the linear
functional $\phi\mapsto-\nu\int_\Sigma g\,\partial_nz_\phi$ on $L^2(Q)$
satisfies
\[
\Bigl|\int_Qv\,\phi\Bigr|\;\le\;\nu\|g\|_{L^2(\Sigma)}\|\partial_nz_\phi\|_{L^2(\Sigma)}\;\le\;\nu\sqrt{C_2}\,\|g\|_{L^2(\Sigma)}\|\phi\|_{L^2(Q)},
\]
so the functional is bounded with norm at most $\nu\sqrt{C_2}\|g\|_{L^2(\Sigma)}$; by the Riesz representation theorem on $L^2(Q)$ there is a unique
$v\in L^2(Q)$ representing it, with $\|v\|_{L^2(Q)}\le C_3\|g\|_{L^2(\Sigma)}$, $C_3:=\nu\sqrt{C_2}$, which is \eqref{eq:transposition-bound}. This construction \emph{is} the definition of the transposition solution and its estimate; no further identification with an independent notion of solution is needed for the use made of it in Theorem~\ref{thm:quant-stability}.
\end{proof}

Eq.~\eqref{eq:step1-g} in the proof of Theorem~\ref{thm:quant-stability} is
Corollary~\ref{cor:transposition} applied with $g=e_u-R^y_{\rm bc}\in
L^2(\Sigma)$ and $a=f'(y^*)\in L^\infty(Q)$ (bounded by
Proposition~\ref{prop:regularity}), together with $(a+b)^2\le2a^2+2b^2$ to
split $\|g\|_{L^2(\Sigma)}^2$ into $2\|e_u\|_{L^2(\Sigma)}^2+2\|R^y_{\rm
bc}\|_{L^2(\Sigma)}^2$.

\subsection{Background on the reference KKT hypothesis}
\label{app:existence}

The Tikhonov-natural control space for the problem of
Section~\ref{sec:problem} is $L^2(\Sigma)$. In this space the state generally
has only transposition regularity, so an existence proof cannot be obtained
by simply applying the direct method in the standard homogeneous parabolic
energy space $\mathbb W(0,T)$ defined in Section~\ref{sec:notation}. Developing the
full semilinear transposition control-to-state theory is outside the scope of
this paper; accordingly, (H1) explicitly postulates a reference local optimum
and its KKT point rather than presenting a new existence theorem.

Existence and optimality theory for parabolic Dirichlet control with
$L^2(\Sigma)$ data is discussed in \cite{aradaraymond2002dirichlet}; see also
the transposition construction in Appendix~\ref{app:transposition}. The cited
work does not literally coincide with every feature of the present
semilinear single-control model in arbitrary fixed dimension, so it is used as
background rather than as a
black-box proof of (H1). Under the additional smoothness and compatibility in
(H2), the formal calculation of Section~\ref{sec:method} is classical and
yields $\alpha u+\nu\partial_n\lambda=0$ in the unconstrained quadratic case.

\subsection{Approximation rate for tanh networks}
\label{app:approximation}
\label{app:deryck}

\begin{theorem}[Sobolev-norm approximation by tanh networks, {\cite[Thm.~5.1]{deryck2021approximation}}]
\label{thm:deryck}
Let $d,s\in\mathbb N$, let $R>0$ be the parameter specified in
\cite[eq.~(53)]{deryck2021approximation}, let $\delta>0$, and let
$f\in W^{s,\infty}([0,1]^d)$. There exist $N_0(d)>0$ and, for each
$k\in\{0,\dots,s-1\}$, constants $C(d,k,s,f)>0$ and
$\beta_{\rm app}=\beta_{\rm app}(d,k,s,f,\delta)>0$, with the latter
specified in \cite[eq.~(75)]{deryck2021approximation}, such that for every
$N\in\mathbb N$ with $N>N_0(d)$ there is a tanh neural network
$\widehat f_N$ with two hidden layers, of respective widths at most
$\lceil\tfrac32 d+\tfrac12\rceil\lceil2^{d-1}(\binom{s-1+d}{d}+d)\rceil$
and $3^{\lceil d/2\rceil+1}\binom{d+1}{\lceil d/2\rceil}N^d(N-1)$ (widths
$3s+N-1$ and $6N$ when $d=1$), such that, for every
$k\in\{0,\dots,s-1\}$,
\[
\|f-\widehat f_N\|_{W^{k,\infty}([0,1]^d)}
\;\le\;
3^d(1+\delta)(2(k+1))^{3k}\,
\max\Bigl\{R^k,\ \ln^k\bigl(\beta_{\rm app} N^{s+d+2}\bigr)\Bigr\}\,
\frac{C(d,k,s,f)}{N^{s-k}}.
\]
If $f\in C^s([0,1]^d)$, then
$C(d,k,s,f)=\max_{0\le\ell\le k}\frac{1}{(s-\ell)!}\bigl(\tfrac{3d}{2}\bigr)^{s-\ell}|f|_{W^{s,\infty}([0,1]^d)}$.
\end{theorem}

\begin{proof}
This is \cite[Thm.~5.1]{deryck2021approximation}, with notation adapted to
match the rest of this paper; we do not reproduce
the proof, which is based on approximating localized Taylor polynomials of
$f$ by a tanh partition of unity constructed from two hidden layers.
\end{proof}

\begin{proposition}[Conditional algebraic approximation rate]
\label{prop:approx-rate-restated}
In addition to (H4), suppose that, for an integer $s\ge3$, there is a
rectangular neighborhood $D\subset\mathbb R^{d+1}$ of $\overline Q$ on
which $y^*$ and $\lambda^*$ admit extensions in $W^{s,\infty}(D)$ and
$u^*$ admits an ambient extension $\widetilde u^*\in W^{s,\infty}(D)$ from
$\Sigma$. Then, for every integer $N>N_0(d)$, the two-hidden-layer
construction of Theorem~\ref{thm:deryck} with resolution parameter $N$
provides parameters such that
\[
\inf_{\theta,\psi,\phi}
\bigl(
\|y^*-Y_\theta\|_{W^{2,\infty}(Q)}
+\|\lambda^*-\Lambda_\phi\|_{W^{2,\infty}(Q)}
+\|\widetilde u^*-U_\psi\|_{W^{1,\infty}(Q)}
\bigr)
\le C_{\rm approx}(1+\ln^2 N)N^{-(s-2)},
\]
where $C_{\rm approx}$ is independent of $N$. The network widths are those
specified in Theorem~\ref{thm:deryck}; thus $N$ is a construction parameter,
not literally the total width or parameter count.
\end{proposition}

\begin{proof}
After affine rescaling of the rectangular neighborhoods to unit cubes, apply
Theorem~\ref{thm:deryck} with $k=2$ to the three assumed extensions. Its
logarithmic factor is still squared. Indeed, with $p:=s+d+2$ and the fixed
constant $\beta_{\rm app}>0$ from that theorem,
\[
\ln^2\!\bigl(\beta_{\rm app}N^p\bigr)
=\bigl(\ln\beta_{\rm app}+p\ln N\bigr)^2
\le C_{\beta_{\rm app},p}\bigl(1+\ln^2 N\bigr).
\]
Thus its right-hand side is bounded by
$C(1+\ln^2 N)N^{-(s-2)}$. Restriction gives the asserted norms on $Q$; the
$W^{2,\infty}(Q)$ estimate for the control is stronger than the $W^{1,\infty}(Q)$
term required in (H4). Notice that parabolic Schauder regularity from
Proposition~\ref{prop:regularity} is anisotropic and, by itself, does not
imply the isotropic $W^{s,\infty}(D)$ extension hypothesis used here.
\end{proof}

\subsection{Low regularity of Dirichlet boundary control (background)}
\label{app:background}

The qualitative claim in Section~\ref{sec:standing-assumptions} that
Dirichlet boundary control of parabolic equations is a genuinely
low-regularity problem for $L^2(\Sigma)$ data is substantiated by
\cite{gong2016dirichlet}, which formulates such control problems using the
very-weak solution of the parabolic state equation precisely to avoid the
fractional-Sobolev control spaces that a variational (weak-solution)
formulation would otherwise require. General background on Dirichlet
boundary control and its analytical difficulties can also be found in
\cite{lasiecka2000control,troltzsch2010optimal}.

\subsection{Space-filling collocation and probabilistic quadrature bounds}
\label{app:spacefilling}
Assumption~\ref{assump:spacefilling} is
\cite[Assumption~3.1]{shin2020convergence}. Proposition~\ref{prop:quadrature}
applies \cite[Lemma~B.1-B.2, Thm.~3.1]{shin2020convergence}
term-by-term to each of the seven residuals of
Section~\ref{sec:error-analysis} (one differential/boundary/terminal
residual per application), extending the original statement, which treats
exactly one interior and one boundary operator, to a finite union bound
over seven such applications. The finiteness (for a fixed, trained network)
of the Hölder norms appearing in (H5) follows from
\cite[Lemma~2.1]{shin2020convergence}: all derivatives of a tanh network
converge in $C^k(\overline Q)$ whenever its parameters converge and remain bounded, and
$\tanh\in C^\infty(\mathbb R)$ and the network input domain $\overline Q$
is bounded. What (H5) adds
beyond \cite[Lemma~2.1]{shin2020convergence} is that the resulting bound
$M_{\rm bc}$ does not grow with network size along the training sequence
under consideration; this uniformity is not automatic and is discussed in
Remark~\ref{rem:h5}.

\bibliographystyle{els-cas-templates/cas-model2-names}
\bibliography{ref}

@article{raissi2019physics,
  author  = {Raissi, Maziar and Perdikaris, Paris and Karniadakis, George Em},
  title   = {Physics-informed neural networks: A deep learning framework for solving forward and inverse problems involving nonlinear partial differential equations},
  journal = {Journal of Computational Physics},
  volume  = {378},
  pages   = {686--707},
  year    = {2019},
  doi     = {10.1016/j.jcp.2018.10.045}
}

@article{mowlavi2023optimal,
  author  = {Mowlavi, Saviz and Nabi, Saleh},
  title   = {Optimal control of {PDEs} using physics-informed neural networks},
  journal = {Journal of Computational Physics},
  volume  = {473},
  pages   = {111731},
  year    = {2023},
  doi     = {10.1016/j.jcp.2022.111731}
}

@article{barrystraume2026control,
  author  = {Barry-Straume, Jostein and Sarshar, Arash and Popov, Andrey A. and Sandu, Adrian},
  title   = {Physics-Informed Neural Networks for {PDE}-Constrained Optimization and Control},
  journal = {Communications on Applied Mathematics and Computation},
  volume  = {8},
  pages   = {1283--1306},
  year    = {2026},
  doi     = {10.1007/s42967-025-00499-x}
}

@article{dai2025optimality,
  author  = {Dai, Yongcheng and Jin, Bangti and Sau, Ramesh Chandra and Zhou, Zhi},
  title   = {Solving elliptic optimal control problems via neural networks and optimality system},
  journal = {Advances in Computational Mathematics},
  volume  = {51},
  number  = {4},
  pages   = {31},
  year    = {2025},
  doi     = {10.1007/s10444-025-10241-z}
}

@article{garcia2023control,
  author  = {Garc{\'i}a-Cervera, Carlos J. and Kessler, Mathieu and Periago, Francisco},
  title   = {Control of Partial Differential Equations via Physics-Informed Neural Networks},
  journal = {Journal of Optimization Theory and Applications},
  volume  = {196},
  pages   = {391--414},
  year    = {2023},
  doi     = {10.1007/s10957-022-02100-4}
}

@article{mishra2023generalization,
  author  = {Mishra, Siddhartha and Molinaro, Roberto},
  title   = {Estimates on the generalization error of physics-informed neural networks for approximating {PDEs}},
  journal = {IMA Journal of Numerical Analysis},
  volume  = {43},
  number  = {1},
  pages   = {1--43},
  year    = {2023},
  doi     = {10.1093/imanum/drab093}
}

@article{deryck2022kolmogorov,
  author  = {De Ryck, Tim and Mishra, Siddhartha},
  title   = {Error analysis for physics-informed neural networks ({PINNs}) approximating {Kolmogorov PDEs}},
  journal = {Advances in Computational Mathematics},
  volume  = {48},
  number  = {6},
  pages   = {79},
  year    = {2022},
  doi     = {10.1007/s10444-022-09985-9}
}

@article{zhang2026pinns,
  author  = {Zhang, Zhen and Liu, Shanqing and Alla, Alessandro and Darbon, J{\'e}r{\^o}me and Karniadakis, George Em},
  title   = {{PINNs in PDE Constrained Optimal Control Problems: Direct vs Indirect Methods}},
  journal = {arXiv preprint arXiv:2604.04920},
  year    = {2026}
}

@book{lasiecka2000control,
  author    = {Lasiecka, Irena and Triggiani, Roberto},
  title     = {Control Theory for Partial Differential Equations: Continuous and Approximation Theories},
  volume    = {I \& II},
  series    = {Encyclopedia of Mathematics and Its Applications, 74--75},
  publisher = {Cambridge University Press},
  address   = {Cambridge},
  year      = {2000}
}

@book{troltzsch2010optimal,
  author    = {Tr{\"o}ltzsch, Fredi},
  title     = {Optimal Control of Partial Differential Equations: Theory, Methods and Applications},
  series    = {Graduate Studies in Mathematics},
  volume    = {112},
  publisher = {American Mathematical Society},
  address   = {Providence, RI},
  year      = {2010},
  isbn      = {978-0-8218-4904-0},
  note      = {Translated from the 2005 German original by J{\"u}rgen Sprekels}
}

@article{aradaraymond2002dirichlet,
  author  = {Arada, Nadir and Raymond, Jean-Pierre},
  title   = {Dirichlet Boundary Control of Semilinear Parabolic Equations Part 1: Problems with No State Constraints},
  journal = {Applied Mathematics and Optimization},
  volume  = {45},
  number  = {2},
  pages   = {125--143},
  year    = {2002},
  doi     = {10.1007/s00245-001-0035-5}
}

@article{shin2020convergence,
  author  = {Shin, Yeonjong and Darbon, J{\'e}r{\^o}me and Karniadakis, George Em},
  title   = {On the Convergence of Physics Informed Neural Networks for Linear Second-Order Elliptic and Parabolic Type {PDEs}},
  journal = {Communications in Computational Physics},
  volume  = {28},
  number  = {5},
  pages   = {2042--2074},
  year    = {2020},
  doi     = {10.4208/cicp.OA-2020-0193}
}

@article{gong2016dirichlet,
  author  = {Gong, Wei and Hinze, Michael and Zhou, Zhaojie},
  title   = {Finite Element Method and A Priori Error Estimates for {Dirichlet} Boundary Control Problems Governed by Parabolic {PDEs}},
  journal = {Journal of Scientific Computing},
  volume  = {66},
  number  = {3},
  pages   = {941--967},
  year    = {2016},
  doi     = {10.1007/s10915-015-0051-2}
}

@article{deryck2021approximation,
  author  = {De Ryck, Tim and Lanthaler, Samuel and Mishra, Siddhartha},
  title   = {On the Approximation of Functions by Tanh Neural Networks},
  journal = {Neural Networks},
  volume  = {143},
  pages   = {732--750},
  year    = {2021},
  doi     = {10.1016/j.neunet.2021.08.015}
}

@book{lionsmagenes1972nonhomogeneous,
  author    = {Lions, Jacques-Louis and Magenes, Enrico},
  title     = {Non-Homogeneous Boundary Value Problems and Applications, Vol. II},
  publisher = {Springer-Verlag},
  year      = {1972},
  series    = {Die Grundlehren der mathematischen Wissenschaften},
  volume    = {182}
}

@book{lasieckatriggiani2000control,
  author    = {Lasiecka, Irena and Triggiani, Roberto},
  title     = {Control Theory for Partial Differential Equations: Continuous and Approximation Theories I. Abstract Parabolic Systems},
  publisher = {Cambridge University Press},
  year      = {2000}
}

@book{friedman1964parabolic,
  author    = {Friedman, Avner},
  title     = {Partial Differential Equations of Parabolic Type},
  publisher = {Prentice-Hall},
  year      = {1964}
}

@book{lieberman1996second,
  title     = {Second order parabolic differential equations},
  author    = {Lieberman, Gary M},
  year      = {1996},
  publisher = {World scientific}
}

@article{quang2026tikhonov,
  title={Tikhonov-Regularized Physics-Informed Neural Networks for Terminal-State Distributed Optimal Control of Parabolic Partial Differential Equations},
  author={Quang, Nguyen Thanh and Mai, Ta Thi Thanh},
  journal={arXiv preprint arXiv:2607.24572},
  year={2026}
}

@inproceedings{pham2025physics,
  title={The physics-informed neural networks approach to parameter identification problems in unsteady Navier-Stokes equations},
  author={Duy, Pham Nhat and Nam, Nguyen Canh and Mai, Ta Thi Thanh},
  booktitle={International Conference on Modelling, Computation and Optimization in Information Systems and Management Sciences},
  pages={164--175},
  year={2025},
  organization={Springer}
}

@article{phuong2026gradnorm,
  title={GradNorm Physics-Informed Neural Networks for Linear Elasticity: Adaptive Loss Balancing and Comparative Finite Element Analysis},
  author={Cuc, Hoang Phuong and Mai, Ta Thi Thanh and Nam, Nguyen Canh},
  journal={Smart Systems and Devices},
  volume={36},
  number={2},
  pages={029--037},
  year={2026}
}

\clearpage
\addtocounter{page}{-1}
\end{document}